\documentclass[11pt, twoside]{article}
\usepackage{amsfonts}
\usepackage{ulem}

\usepackage{amssymb}
\usepackage{amsmath}
\usepackage{amsthm}
\usepackage{xcolor}
\usepackage{mathrsfs}
\usepackage{enumerate}

\allowdisplaybreaks

\allowdisplaybreaks

\def\rr{{\mathbb R}}
\def\R{{\mathbb R}}
\def\rn{{{\rr}^n}}
\def\zz{{\mathbb Z}}
\def\Z{{\mathbb Z}}

\def\nn{{\mathbb N}}
\def\N{{\mathbb N}}
\def\fz{\infty}

\def\ccc{{\mathbb C}}

\def\cs{{\mathcal S}}

\def\cp{{\mathcal P}}
\def\az{\alpha}

\def\supp{{\rm{\,supp\,}}}

\def\loc{{\rm{\,loc\,}}}

\def\ls{\lesssim}
\def\lz{\lambda}

\def\vz{\varphi}

\def\wz{\widetilde}

\def\hs{\hspace{0.3cm}}

\def\r{\right}
\def\lf{\left}

\def\bint{{\ifinner\rlap{\bf\kern.30em--}
\int\else\rlap{\bf\kern.35em--}\int\fi}\ignorespaces}

\def\sbint{{\ifinner\rlap{\bf\kern.32em--}
\hspace{0.078cm}\int\else\rlap{\bf\kern.45em--}\int\fi}\ignorespaces}

\newtheorem{theorem}{Theorem}[section]
\newtheorem{lemma}[theorem]{Lemma}

\newtheorem{proposition}[theorem]{Proposition}
\theoremstyle{definition}
\newtheorem{example}[theorem]{Example}
\newtheorem{remark}[theorem]{Remark}

\newtheorem{definition}[theorem]{Definition}
\numberwithin{equation}{section}

\numberwithin{equation}{section}

\numberwithin{equation}{section}

\begin{document}

\arraycolsep=1pt

\title{\Large\bf Variable Anisotropic Singular Integral Operators\footnotetext{\hspace{-0.35cm} {\it 2010
Mathematics Subject Classification}. {42B35, 42B30, 46A32, 32A55.}
\endgraf{\it Key words and phrases.} anisotropic Hardy space, continuous ellipsoid cover,  maximal function, singular integral operator.
\endgraf  Marcin Bownik was partially supported by NSF grants DMS-1665056 and DMS-1956395. Baode Li was supported by the National Natural Science Foundation of China (Grant Nos. 11861062, 11661075 \& 11561065).
A majority of the work was done when Baode Li and Jinxia Li were visiting University of Oregon in the academic year 2018--19.
\endgraf $^\ast$\,Corresponding author.
}}
\author{Marcin Bownik, Baode Li, and Jinxia Li$^\ast$}
\date{ }
\maketitle

\vspace{-0.8cm}

\begin{center}
\begin{minipage}{13cm}\small
{\noindent{\bf Abstract.}
We introduce the class of variable anisotropic singular integral operators associated to a continuous multi-level ellipsoid cover $\Theta$ of $\mathbb{R}^n$ introduced by Dahmen, Dekel, and Petrushev \cite{ddp}. This is an extension of the classical isotropic singular integral operators on $\R^n$ of arbitrary smoothness and their anisotropic analogues for general expansive matrices introduced by the first author \cite{b}. We establish the boundedness of variable anisotropic singular integral operators $T$ on the Hardy spaces with pointwise variable anisotropy $H^p(\Theta)$, which were developed by Dekel, Petrushev, and Weissblat \cite{dpw}. In contrast with the general theory of Hardy spaces on spaces of homogenous type, our results work in the full range $0<p\leq 1$.}
\end{minipage}
\end{center}

%%%%%%%%%%%%%%%%%%%%%%%%%%%%%%%%%%%%%%%%%%%%%%%%%%%%%%%%%%%%%%%%%%%%%%%%%
%%%%%%%%%%%%%%%%%%%%%%%%%%%  1.  Preliminaries  %%%%%%%%%%%%%%%%%%%%%%%%%
%%%%%%%%%%%%%%%%%%%%%%%%%%%%%%%%%%%%%%%%%%%%%%%%%%%%%%%%%%%%%%%%%%%%%%%%%

\section{Introduction\label{s1}}

Calder\'on-Zygmund operators play an important role in harmonic analysis on $\R^n$  and are the central object of study. They are bounded not only on Lebesgue $L^p(\rn)$ spaces for $1<p<\infty$, but also on its natural extension for $0<p\le 1$, the Hardy $H^p(\rn)$ spaces. While Hardy spaces were initially defined in the complex variable setting, they were extended to real-variable setting in the celebrated works of Stein and Weiss \cite{sw} and Fefferman and Stein \cite{fs72}. Since then, Hardy spaces have been studied in different settings and domains. Among the most general setting, where Hardy spaces are studied, are certain metric measure spaces known as spaces as homogeneous type, which were introduced by Coifman and Weiss \cite{cw71, cw77}. However, due to lack of higher order smoothness and vanishing moments, such spaces can be meaningfully defined only when $p$ is close to 1, see \cite{am, hs, hhllyy}.

In this paper we are interested in developing results for Hardy spaces on $\R^n$ which hold for the entire range of $0<p\le 1$. This includes classical isotropic Hardy spaces of Fefferman and Stein \cite{fs72}, parabolic Hardy spaces of Calder\'on and Torchinsky \cite{ct75, ct77}, and anisotropic Hardy space associated with expansive matrices  \cite{b} which were also studied in \cite{bb,bd,blyz,h,w,zl}.
However, the most general class of (unweighted) Hardy spaces on $\R^n$ defined for the entire range of $0<p\le 1$ are spaces with pointwise variable anisotropy developed by Dekel, Petrushev, and Weissblat \cite{dpw}. More precisely, they correspond to the largest class of spaces of homogeneous type on $\R^n$ equipped with Lebesgue measure for which Hardy spaces were developed for all $0<p\le 1$. Roughly speaking they match with quasi-distances on $\R^n$ for which balls are equivalent to ellipsoids. More precisely, these spaces are determined by continuous ellipsoid covers of $\R^n$, which were introduced and studied by Dahmen, Dekel, and Petrushev \cite{ddp}.

Several results for spaces with pointwise variable anisotropy were shown including grand maximal function characterization, atomic decomposition, and classification of Hardy spaces \cite{dpw}, the duality of Hardy spaces \cite{dw}, and molecular decomposition \cite{abr}. These results are generalizations of well-known results for Hardy spaces in the classical isotropic setting of $\R^n$, parabolic setting, and anisotropic setting. However, an important missing ingredient in the setting of continuous ellipsoid covers is a satisfactory definition of Calder\'on-Zygmund operators of arbitrary smoothness.
In this paper we show this remaining link by providing the definition of Calder\'on-Zygmund operators which is the extension of the class of operators in the anisotropic setting \cite{b} and at the same type in the setting of spaces of homogeneous type \cite{cw71}.

The theory of singular integral operators plays an important role in harmonic analysis and partial differential equations; see, for example, \cite{g,g1,s}. In the classical isotropic setting of $\R^n$ we consider Calder\'on-Zygmund operators $T$ with regularity $s$ of the form
\[
Tf(x) = \int_{\R^n} K(x,y) f(y), \qquad x \not\in \supp f,\  f\in C^\infty_c(\rn),
\]
whose kernel $K(x,y)$ satisfies the bound
\begin{equation}\label{rd}
|\partial^\alpha_y  K(x,y) | \le C |x-y|^{-n-|\alpha|}
\qquad\text{for all }x\ne y \text{ and multi-indices } |\alpha|\le s.
\end{equation}
It is well-known that operators $T$ are bounded on isotropic Hardy spaces $H^p(\R^n)$ provided that $s>n(1/p-1)$ and $T$ preserves vanishing moments $T^*(x^\alpha)=0$ for $|\alpha|<s$, see  \cite[Proposition 7.4.4]{mc97}, \cite[Theorem III.4]{s}. The first author \cite{b} introduced anisotropic Calder\'{o}n-Zygmund operators associated with expansive dilations and has shown their boundedness on anisotropic Hardy spaces, where the anisotropy is fixed and global on $\rn$. An extension of these results to product anisotropic Hardy spaces was done in \cite{lbyz}. In the context of Hardy spaces with pointwise variable anisotropy \cite{dpw} we introduce the following class of singular integral operators adapted to variable anisotropy depending on a point $x$ in $\rn$ and a scale $t$ in $\mathbb{R}$.

Suppose that $\Theta$ is a continuous ellipsoid cover consisting of ellipsoids $\theta_{x,t}$ with center $x\in\R^n$ and scale $t\in \R$ of the form $\theta_{x,t}=M_{x,t}(\mathbb B^n)+x$, where $M_{x,t}$ is an invertible matrix and $\mathbb B^n$ is the unit ball in $\R^n$, see Definition \ref{d2.1}. An ellipsoid cover $\Theta$ defines a spaces of homogeneous type \cite{ddp} with quasi-distance $\rho_\Theta$ defined as infimum of ellipsoid volumes
\[
\rho_\Theta(x,y):=\inf_{\theta\in\Theta}\lf\{|\theta|:x,y\in\theta\r\}.
\]
An anisotropic analogue of the bound \eqref{rd} takes the form
\begin{equation}\label{rda}
\lf|\partial^\alpha_y[K(\cdot,M_{y,\,m}\cdot)](x,M^{-1}_{y,\,m}y)\r|\le C/\rho_\Theta(x,y) \qquad\text{for all }x\ne y \text{ and multi-indices } |\alpha|\le s,
\end{equation}
where  $m=  -  \log_{2} \rho_\Theta(x,y)$. It can be shown that \eqref{rda} is a generalization of \eqref{rd} when $\Theta$ is the isotropic cover consists of balls $\theta_{x,t}=2^{-t}\mathbb B^n+x$. Moreover, \eqref{rda} with regularity $s=1$ implies the well-known estimates in the setting of spaces of homogeneous type
\begin{align*}
|K(x,y)| &\le C/\rho_\Theta(x,y),
\\
|K(x,y)-K(x,y')| &\le C \frac{[\rho_\Theta(y,y')]^\delta}{[\rho_\Theta(x,y)]^{1+\delta}}
\qquad\text{if \ } \rho_\Theta(y,y') \le \frac{1}{2\kappa} \rho_\Theta(x,y),
\end{align*}
where $\kappa$ is the triangle inequality constant of $\rho_\Theta$.
The main result of this paper shows the boundedness of variable anisotropic singular integral operators $T$ from $H^p(\Theta)$ to itself and from $H^p(\Theta)$ to Lebesgue spaces $L^p(\rn)$ for the entire range of $0< p\le 1$, provided regularity and vanishing moments of $T$ are met analogous to the isotropic case. This generalizes classical results for isotropic Hardy spaces of Fefferman-Stein \cite{fs72, mc97} and anisotropic Hardy spaces \cite{b}.

The proof of the main theorem is conceptually simple, but technically challenging. The central idea is to show that $T$ maps atoms into uniformly bounded functions in $H^p(\Theta)$ known as molecules. Then, the atomic decomposition of $H^p(\Theta)$ yields the boundedness of $T$. The main technical problem with this argument is that boundedness on atoms does not necessarily imply boundedness of $T$  unless equivalence of finite and infinite atomic decompositions is shown \cite{b2, msv}. Instead,  motivated by the paper of Huang, Liu, Yang, and Yuan \cite{hlyy}, we improve the Calder\'on-Zygmund decomposition of $H^p(\Theta)$ by showing that atomic decomposition of any $f\in H^p(\Theta)\cap L^q(\rn)$ also converges in $L^q(\rn)$ norm for $1<q<\infty$.

%%%%%%%

This paper is organized as follows.
In Section \ref{s2}, we first recall notation, definitions, and properties of  continuous ellipsoid cover $\Theta$ and quasi-distance $\rho_\Theta$ that are used throughout the paper. In Section \ref{s3} we define the Hardy space $H^p(\Theta)$ by means of the radial grand maximal function and  the nontangential grand maximal function with arbitrary aperture and recall its characterization by atomic decompositions.  In the next section we show technical improvements in the Calder\'on-Zygmund decomposition and the atomic decomposition of variable anisotropic Hardy space $H^p(\Theta)$, which were originally established by Dekel, Petrushev, and Weissblat \cite{dpw}.
Section \ref{s5} is devoted to variable anisotropic singular integral operators (VASIOs). We show that VASIOs are indeed an extension of the classical isotropic singular integral operators on $\R^n$ of arbitrary smoothness and their anisotropic analogues for general expansive matrices. Finally, in Section \ref{s6} we prove main theorems by showing that $T$ is bounded from $H^p(\Theta)$ to $L^p(\rn)$ and bounded from $H^p(\Theta)$ to itself.

Finally, we make some conventions on notation.
Let $\nn:=\{1,\, 2,\,\ldots\}$ and $\nn_0:=\{0\}\cup\nn$.
For any $\az:=(\az_1,\ldots,\az_n)\in\nn_0^n$,
$|\az|:=\az_1+\cdots+\az_n$ and $\partial^\az:=
(\frac{\partial}{\partial x_1})^{\az_1}\cdots
(\frac{\partial}{\partial x_n})^{\az_n}$.
Throughout the whole paper, we denote by $C$ a positive
constant which is independent of the main parameters, but it may vary from line to line. The symbol $D\ls F$ means that $D\le CF$. If $D\ls F$ and $F\ls D$, we then write $D\sim F$. For any sets $E,\,F \subset \rn$, we use $E^\complement$ to denote the set $\rn\setminus E$. Let $\cs$ be the space of Schwartz functions, $\cs'$ the space of tempered distributions, and $C^N$ the space of  continuously differentiable functions of order $N$.

%%%%%%%%%%%%%%%%%%%%%%%%%%%%%%%%%%%%%%%%%%%%%%%%%%%%%%%%%%%%%%%%%%%%%

%%%%%%%%%%%%%%%%%%%%%%%  section 2  %%%%%%%%%%%%%%%%%%%%%%%%%%%%%%%%%

%%%%%%%%%%%%%%%%%%%%%%%%%%%%%%%%%%%%%%%%%%%%%%%%%%%%%%%%%%%%%%%%%%%%%

\section{Anisotropic Continuous Ellipsoid Covers of $\rn$ \label{s2}}

In this section we recall the properties of continuous ellipsoid covers which were originally introduced by Dahmen, Dekel, and Petrushev \cite{ddp}. An {\it ellipsoid} $\xi$ in $\rn$ is an image of the Euclidean unit ball $\mathbb{B}^n:=\{x\in\rn: |x|<1\}$ under an affine transform, i.e.,
$$\xi:=M_\xi(\mathbb{B}^n)+c_\xi,$$
where $M_\xi$ is an invertible matrix and $c_\xi$ is the center. For any ellipsoid $\xi$ and $\lambda>0$, define a dilated ellipsoid $\lambda\xi$ by $$\lambda\xi:=\lambda M_\xi(\mathbb{B}^n)+c_\xi.$$

We begin with the definition of continuous ellipsoid covers, which was introduced in  \cite[Definition 2.4]{ddp}.

\begin{definition}\label{d2.1}
We say that
$$\Theta:=\{ \theta_{x,\,t}: x\in\rn,t\in\mathbb{R}\}$$
is a {\it continuous ellipsoid cover} of $\rn$, or shortly a cover,
if there exist constants $\mathbf{p}(\Theta):=\{a_1,\ldots, a_6\}$ such that:
\begin{itemize}
\item[(i)]
For every $x\in \rn$ and $t\in \mathbb{R}$, there exists an ellipsoid $\theta_{x,\,t}:=M_{x,\,t}(\mathbb{B}^n)+x$, where $M_{x,\,t}$ is an invertible matrix and $x$ is the center, satisfying
\begin{equation}\label{e2.1}
a_12^{-t}\leq|\theta_{x,\,t}|\leq a_2 2^{-t}.
\end{equation}
\item[(ii)]
Intersecting ellipsoids from $\Theta$ satisfy ``shape condition'', i.e., for any $x,\,y\in \rn$, $t\in \mathbb{R}$ and $s\ge0$, if $\theta_{x,\,t}\cap \theta_{y,\,t+s}\ne\emptyset $, then

\begin{equation}\label{e2.2}
a_3 2^{-a_4 s}\le 1 / \| (M_{y,\,t+s})^{-1} M_{x,\,t}\|
\le \|(M_{x,\,t})^{-1} M_{y,\,t+s}\|
\le a_5 2^{-a_6 s}.
\end{equation}
Here, $\|\cdot \|$ is the matrix norm of $M$ given by $\|M\|:=\max_{|x|=1}|Mx|$.
\end{itemize}
\end{definition}

The word {\it continuous} refers to the fact that ellipsoids $\theta_{x,t}$ are defined for all values of $x\in \R^n$ and $t\in \R$. In contrast, a {\it discrete ellipsoid cover} is indexed over integer scales $t\in \Z$ with discrete choice of centers $x\in \mathcal D_t \subset \R^n$, which satisfy some additional conditions, see \cite[Definition 2.1]{ddp}, such as $\bigcup_{x\in \mathcal D_t} \theta_{x,t} =\R^n$.

It is worth adding that for our purposes it is not necessary to assume any measurability or continuity condition on a continuous ellipsoid cover $\Theta$. We say that an ellipsoid cover
$\Theta$ is {\it pointwise continuous} if for every $t\in \R$, the matrix valued function $x \mapsto M_{x,t}$ is continuous. That is,
\begin{equation}\label{e2.0}
\| M_{x',t}-M_{x,t}\|\rightarrow 0 \ \ {\rm as}\ \ x'\rightarrow x.
\end{equation}
The condition \eqref{e2.0} is implicitly used in \cite{dpw} to guarantee that the superlevel set $\Omega$  corresponding to the grand maximal function, which is given by \eqref{omega}, is open. However, as we will see this assumption is not necessary since it is always possible to construct an equivalent ellipsoid cover
\[
\Xi:=\{ \xi_{x,\,t}: x\in\rn,t\in\mathbb{R}\}
\]
such that $\Xi$ is pointwise continuous and $\Xi$ is equivalent to $\Theta$. We say that two ellipsoid covers $\Theta$ and $\Xi$ are {\it equivalent} if there exists a constant $C>0$ such that for any  $x\in\rn$ and $t\in\mathbb {R}$, we have
\begin{equation}\label{e2.x7}
\frac1{C} \xi_{x,t} \subset \theta_{x,t} \subset C \xi_{x,t}.
\end{equation}

\begin{theorem}\label{t2.1}
Given an ellipsoid cover $\Theta$, there exists an equivalent ellipsoid cover $\Xi$, which is pointwise continuous.
\end{theorem}

To prove Theorem \ref{t2.1}, we need the following lemmas.

\begin{lemma}\label{gal}
Suppose $A$ and $B$ are $n\times n$ positive definite matrices. Let $c>0$ be a constant. The following are equivalent:
\begin{enumerate}[\rm (i)]
\item $||A^{-1} B || \le c$,
\item $A^{-2} \le c^2 B^{-2}$, where $\le$ denotes the partial order among hermitian matrices,
\item $B(\mathbb B^n) \subset  A(c\mathbb B^n)$.
\end{enumerate}
\end{lemma}

\begin{proof}
(i) is equivalent to $A^{-1}B(\mathbb B^n) \subset c \mathbb B^n$, which is equivalent to (iii). (i) is also equivalent to
\[
\langle A^{-2} v , v \rangle = ||A^{-1}v||^2 \le c^2 ||B^{-1}v||^2 = c^2 \langle B^{-2}v,v \rangle
\qquad\text{for all }v\in \R^n,
\]
which shows the equivalence with (ii).
\end{proof}

\begin{lemma}\label{l2.x2}
For any ellipsoid cover $\Theta$ and fixed $t\in\mathbb{R}$, there exists a bounded continuous function $r: \rn \to (0,\infty)$ such that balls $B(x,r(x))\subset \theta_{x,t}$ for all $x\in \R^n$.
\end{lemma}

\begin{proof}
Fix $t\in\mathbb{R}$. For $x\in\rn$ let $r_x:=\|M^{-1}_{x,t}\|^{-1}$. Since $\|M^{-1}_{x,t}\|r_x=1$, we have $M^{-1}_{x,t}r_x(\mathbb{B}^n)\subset \mathbb{B}^n$ and hence
\begin{equation}\label{e2.x6}
B(x,r_x)\subset\theta_{x,t}=x+M_{x,t}(\mathbb B^n).
\end{equation}
This together with the shape condition implies that, for any $x'\in B(x,r_x)$,
$$
\|M^{-1}_{x',t}\|
\le\|M^{-1}_{x',t}M_{x,t}\|\|M^{-1}_{x,t}\|\le a_5\|M^{-1}_{x,t}\|.
$$
Hence, $\frac1{a_5}r_x=\frac1{a_5}\|M^{-1}_{x,t}\|^{-1}\le \|M^{-1}_{x',t}\|^{-1}=r_{x'}.$ Similarly, we have $r_{x'}\le a_5r_x.$
Therefore,
\begin{equation}\label{e2.x1}
\frac1{a_5}r_x\le r_{x'}\le a_5r_x,\ \ x'\in B(x,r_x).
\end{equation}

By \eqref{e2.1} we have $|r_x|^n\sim|B(x,r_x)|\ls|\theta_{x,t}|\sim2^{-t}$ for any $x\in\rn$. Applying the Vitali covering lemma for the cover
 $\{B(x,\frac15r_x)\}_{x\in \R^n}$, there exists a sequence $\{x_i\}_{i\in \N} $ in $\R^n$ such that the balls $B(x_i,\frac15r_{x_i})$, $i\in \N$, are mutually disjoint and  $\rn=\bigcup_{i=1}^\infty B(x_i,r_{x_i})$. For simplicity we denote $r_i:= r_{x_i}$. For $j\in \N$ we let
 $$I(j):=\{i: B(x_i,r_i)\cap B(x_j,r_j)\neq\emptyset\}.$$
By \eqref{e2.x1} we have
$$B\lf(x_i,\frac15r_i\r)\subset B(x_i,r_i)\subset B(x_j,(2a_5+1)r_j)$$
and hence $\bigcup_{i\in I(j)}B(x_i,\frac15r_i)\subset B(x_j,(2a_5+1)r_j)$. From this and \eqref{e2.x1}, it follows that
\begin{align}\label{e2.x2}
\sharp I(j)
&\le\frac{\sum_{i\in I(j)}|B(x_i,r_i)|}{\min_{i\in I(j)}|B(x_i,r_i)|}\le\frac{\sum_{i\in I(j)}5^n|B(x_i,\frac15r_i)|}{|B(x_j,\frac1{a_5}r_j)|}\\
&\le\frac{5^n|B(x_j,(2a_5+1)r_j)|}{|B(x_j,\frac1{a_5}r_j)|}
=[5a_5(2a_5+1)]^n=:L.\nonumber
\end{align}

Choose a function $\phi\in C^\fz$ such that $\supp\phi\subset \mathbb{B}^n$, $0\le\phi\le 1$ and $\phi\equiv 1$ on $\frac12\mathbb{B}^n$. For every $i\in\nn$, define
$$\phi_i(x):=
\frac{\phi\lf(\frac{x-x_i}{5r_i}\r)}{a_5L}r^\circ_i,$$
where $r^\circ_i:=\min\{r_j:\, B(x_i,r_i)\cap B(x_j,r_j)\neq\emptyset\}$ and $L$ is as in \eqref{e2.x2}. From this and \eqref{e2.x1}, it follows that
\begin{equation}\label{e2.x3}
r^\circ_i\sim|B(x_i,r_i)|^{\frac1n}
\le|\theta_{x_i,t}|^{\frac1n}\sim2^{-\frac{t}{n}}.
\end{equation}
For $x\in \R^n$ we define
$$r(x):=\sum_{i=1}^\infty \phi_i(x).$$
This is a well-defined continuous function since on each ball $B(x_j,r_j)$ since the above series has $\le L$ non-zero terms corresponding to $i\in I(j)$. More precisely, if $x\in B(x_j,r_j)$, then
$$r(x)\le\sum_{i\in I(j) }\phi_i(x)\le\sum_{i\in I(j) }\frac{r^\circ_i}{a_5L}\le\sum_{i\in I(j)}\frac{r_j}{a_5L}\le \frac{r_j}{a_5} \le r_x.$$
This together with \eqref{e2.x6} implies that $B(x,r(x))\subset B(x,r_x)\subset\theta_{x,t}$.
\end{proof}

\begin{proof}[Proof of Theorem \ref{t2.1}]
Without loss of generality, we can assume that the matrices $M_{x,t}$ defining ellipsoids $\theta_{x,t}$ are positive definite. Indeed, it suffices to use the matrix absolute value, which is defined for a matrix $A$ as  $ |A|:=(A^*A)^{1/2}$. It is immediate that $A$ and its absolute value $|A|$ determine the same ellipsoid $A(\mathbb{B}^n)=|A|(\mathbb{B}^n)$. Hence, by Lemma \ref{gal} we can replace matrices $M_{x,t}$ by their absolute values $|M_{x,t}|$, which yields the same ellipsoid cover $\Theta$ satisfying \eqref{e2.1} and \eqref{e2.2}.

Fix $t\in \R$. Let $r:=r_t:\rn\to(0,\,\infty)$ be the continuous function as in Lemma \ref{l2.x2}. Choose a sequence $\{x_k\}$ of points in $\R^n$ such that balls $\bigcup_{k\in \N} B(x_k,r(x_k))=\R^n$. Choose a partition $\{E_k\}_{k\in\N}$ of $\R^n$ into measurable sets such that $E_k \subset B(x_k, r(x_k))$ for all $k\in \N$. For example, define
\[
E_k =
\begin{cases} B(x_1,r(x_1)) & k=1,\\
B(x_k,r(x_k)) \setminus \bigcup_{i=1}^{k-1} B(x_i, r(x_i)) & k\ge 2.
\end{cases}
\]
Define $\wz M_{x,t}=M_{x_k,t}$ if $x\in E_k$ for some $k\in \N$.
Finally, define an ellipsoid cover
$${\Xi:=\{\xi_{x,t}:=N_{x,t}(\mathbb{B}^n)+x:
\,x\in\rn,t\in \mathbb{R}\},}$$
using positive definite matrices
\begin{equation}\label{e2.x5}
N_{x,t}:=\bigg(\frac1{|B(x,r(x))|}
\int_{B(x,r(x))}(\wz M_{y,t})^{-2}\ dy \bigg)^{-1/2}.
\end{equation}
Since the function $r$ is bounded we have
\[
||(\wz M_{y,t})^{-1}|| = ||(M_{x_k,t})^{-1}|| \le 1/r(x_k) \le \sup_{x\in B(y,||r||_\infty)} 1/r(x)
\qquad\text{for all } y\in \R^n,
\]
where $k\in \N$ is such that $y \in E_k$, and hence $x_k \in B(y,||r||_\infty)$. Therefore, the vector-valued integral in \eqref{e2.x5} is well defined with values in positive definite matrices. By the continuity of the function $r$ we can easily show that
\[
x \mapsto \frac1{|B(x,r(x))|}
\int_{B(x,r(x))}(\wz M_{y,t})^{-2}\ dy
\]
is a continuous positive definite matrix valued function. Using the fact that the square root mapping $A \mapsto A^{1/2}$ is continuous on the space of all positive definite $n\times n$ matrices $A$, and the inverse mapping $A \mapsto A^{-1}$ is continuous on the space of $n\times n$ invertible matrices, we deduce that $x\mapsto N_{x,t}$ is also continuous.

It remains to show that the resulting ellipsoid $\Xi$ cover is equivalent with $\Theta$. Fix $x\in \R^n$. Take any $y\in B(x,r(x))$ and let $k\in \N$ be such that $y\in E_k$. Since $y\in \theta_{x,t} \cap \theta_{x_k,t} \ne \emptyset$ by the shape condition and Lemma \ref{gal} we have
\[
(a_5)^{-2} (M_{x,t})^{-2} \le (M_{x_k,t})^{-2} \le (a_5)^2 (M_{x,t})^{-2}.
\]
Hence,
\[
(a_5)^{-2} (M_{x,t})^{-2} \le (\wz M_{y,t})^{-2} \le (a_5)^2 (M_{x,t})^{-2}.
\]
Integrating the above inequality over $y\in B(x,r(x))$ yields
\[
(a_5)^{-2} (M_{x,t})^{-2} \le (N_{x,t})^{-2} \le (a_5)^2 (M_{x,t})^{-2}.
\]
Hence, by Lemma \ref{gal} the ellipsoid cover $\Xi$ is equivalent with $\Theta$.
\end{proof}

Next we collect results about ellipsoid covers from \cite{ddp, dhp, dpw} which will be used subsequently.

\begin{lemma}\label{l2.2}
\begin{enumerate}
\item[\rm(i)]Let $\Theta$ be an ellipsoid cover. Then there exists $J:=J(\mathbf{p}(\Theta))\geq1$ such that for any $x\in\rn$ and $t\in\mathbb{R}$,
 $$\theta_{x,\,t}\subset \frac12 \theta_{x,\,t-J}.$$
\item[\rm(ii)] For any $x,y\in \rn$ and  $s,t\in\mathbb R$ with $t \le s$, if $\theta_{x,\,t}\cap\theta_{y,\,s}\neq\emptyset$, there exists a constant $\gamma>0$ such that
    $$\theta_{y,\,s}\subset\theta_{x,\,t-\gamma}.$$
 \end{enumerate}
\end{lemma}
Part (i) of Lemma \ref{l2.2} is shown in \cite[Lemma 2.3]{dpw}. Part (ii) is an easy adaptation of the proof of \cite[Lemma 2.8]{ddp}, see also \cite[Lemma 2.4]{dpw}. Note that by increasing $\gamma$ and $J$ if necessary we can assume that $J=\gamma$. However, we prefer to keep a separate notation for $J$ and $\gamma$ to be consistent with the convention used in \cite{ddp, dpw}. The following is an immediate consequence of Lemma \ref{l2.2}.

\begin{lemma}\label{l4.8}
Let $x\in\rn$, $t\in\mathbb{R}$ and $\gamma$ be as in Lemma \ref{l2.2}(ii). Then, the sets  $\theta_{x,\,t-(j+1)\gamma}\setminus \theta_{x,\,t-j\gamma}$, $j\in\mathbb \Z$, are pairwise disjoint and
\begin{equation}\label{l4.8a}
\rn=\theta_{x,\,t}\cup \bigcup_{j\in\nn_0}(\theta_{x,\,t-(j+1)
    \gamma}\setminus \theta_{x,\,t-j\gamma}),
\end{equation}
\begin{equation}\label{l4.8b}
\rn \setminus\{x\}=\bigcup_{j \in\zz}(\theta_{x,\,t-(j+1)
    \gamma}\setminus \theta_{x,\,t-j\gamma}).
\end{equation}
\end{lemma}

\begin{proof}
By Lemma \ref{l2.2}(ii) we have $\theta_{x,\,t-j\gamma} \subset \theta_{x,\,t-(j+1)\gamma}$ for all $j\in\Z$. Hence, $\theta_{x,\,t-(j+1)\gamma}\setminus \theta_{x,\,t-j\gamma}$, $j\in\mathbb \Z$, are pairwise disjoint. Moreover, by Lemma \ref{l2.2}(i) for any $k\in\nn_0$ we have
\[
\theta_{x,\,t}\subset \frac1{2^k} \theta_{x,\,t-kJ}.
\]
Hence,
\[
\R^n = \bigcup_{k\in\nn_0} 2^k \theta_{x,\,t} \subset \bigcup_{k\in\nn_0}  \theta_{x,\,t-kJ}
= \bigcup_{j\in \nn_0}  \theta_{x,\,t-j\gamma} \subset \R^n,
\]
and the above inclusions are equalities. Likewise,
\[
\theta_{x,\,t+kJ}\subset \frac1{2^k} \theta_{x,\,t}
\]
implies that
\[
\{x\}= \bigcap_{k\in \nn_0} \frac1{2^k} \theta_{x,\,t} = \bigcap_{k\in \nn_0} \theta_{x,\,t+kJ} = \bigcap_{j\in \nn_0} \theta_{x,\,t+j\gamma}.
\]
Hence, \eqref{l4.8a} and \eqref{l4.8b} follow immediately.
\end{proof}

\begin{definition}\label{d2.3}
A $\it{ quasi}$-$\it {distance}$ on a set $X$ is a mapping $\rho:\ X\times X\rightarrow [0,\fz)$ that satisfies the following conditions for all $x,y,z\in X$:
\begin{enumerate}
\item[(i)]
$\rho(x,y)=0\Leftrightarrow x=y$;
\item[(ii)]
$\rho(x,y)=\rho(y,x)$;
\item[(iii)]
For some $\kappa\ge 1$,
$$\rho(x,y)\le\kappa(\rho(x,z)+\rho(y,z)).$$
\end{enumerate}
\end{definition}

Dekel, Han, and Petrushev have shown that an ellipsoid cover $\Theta$ induces a quasi-distance $\rho_\Theta$ on $\R^n$, see \cite[Proposition 2.1]{dhp}. Moreover, $\R^n$ equipped with the quasi-distance $\rho_\Theta$ and the Lebesgue measure is a space of homogeneous type, \cite[Proposition 2.10]{ddp}.

\begin{proposition}\label{p2.4}
Let $\Theta$ be a continuous ellipsoid cover. The function $\rho_\Theta:\ \rn\times \rn\rightarrow [0,\fz)$ defined by
\begin{equation}\label{e2.3}
\rho_\Theta(x,y):=\inf_{\theta\in\Theta}\lf\{|\theta|:x,y\in\theta\r\}
\end{equation}
is a quasi-distance on $\rn$. Moreover, the Lebesgue measure of balls
\begin{equation}\label{e3.3}
B_{\rho_\Theta}(x,r):=\{y\in\rn:\rho_\Theta(x,y)<r\}
\end{equation}
with respect to the quasi-distance $\rho_\Theta$ satisfies
\[
|B_{\rho_\Theta}(x,r)| \sim r \qquad\text{for all }x\in \R^n, \ r>0,
\]
with equivalence constants depending only on $\mathbf{p}(\Theta)$.
\end{proposition}

The following lemma is proved for discrete ellipsoid covers \cite[Theorem 2.1]{dhp} and continuous ellipsoid covers \cite[Theorem 2.9]{dpw}.
\begin{lemma}\label{l2.6}
Let $\Theta$ be a continuous ellipsoid cover and $\rho_\Theta$ induced by \eqref{e2.3}. There exist positive constants $C_0$ and $C_1$  such that
$$C_0|x-z|^{\frac1{a_4}}\le|\rho_\Theta(x,z)|\le C_1|x-z|^{\frac1{a_6}}\ \ \ {\rm for\ all}\ x,z \in\rn\ \ {\rm and}\ \ \rho_\Theta(x,z)\ge 1.$$
\end{lemma}

The following result is stated without the proof in \cite[Theorem 2.7]{dpw}. For completeness, we include its proof.

\begin{proposition}\label{p4.10}
Let $\Theta$ be a continuous ellipsoid cover and let $\rho_\Theta$ be a quasi-distance as in \eqref{e2.3}. For any ball $B_{\rho_\Theta}(x,r)$ with $x\in\rn$ and $r>0$, there exist $t_1, t_2\in \R$ such that
\[
\theta_{x, \,t_1}\subset B_{\rho_\Theta}(x,r)
\subset\theta_{x,\,t_2}
\qquad\text{and}\qquad
|\theta_{x, \,t_1}|\sim |\theta_{x,\,t_2}|\sim r,
\]
where equivalence constants depend only on $\mathbf{p}(\Theta)$.
\end{proposition}

\begin{proof}
By \eqref{e2.3}, any ball $B_{\rho_\Theta}(x,r)$ with respect to quasi-distance $\rho_\Theta$ is a union of all ellipsoids $\theta$ in $\Theta$ that contain $x$ and $|\theta|<r$, i.e.,
\begin{align}\label{e4.9}
B_{\rho_\Theta}(x,r)=\{y\in\rn:\rho_\Theta(x,y)<r \}=\bigcup_{x\in\theta\in\Theta,\, |\theta|<r}\theta.
\end{align}
We claim that there exists an ellipsoid $\theta_{x,\,t_1}\in\Theta$ such that $|\theta_{x,\,t_1}|<r$ and $|\theta_{x,\,t_1}|\sim r$. In fact, by \eqref{e2.1}, we know that for the two constants $\tilde a_1:=\log_2 a_1$ and $\tilde a_2:=\log_2 a_2$, it holds true that
$$2^{-t+\tilde a_1}\le |\theta_{x,\,t}|\le 2^{-t+\tilde a_2} \quad \mathrm{for\ any\ } x\in\rn\ \mathrm{and} \  t\in\mathbb{R}.$$
Hence, by a substitution $\tilde t=t-\tilde a_2$, we have
$$2^{-\tilde t+\tilde a_1-\tilde a_2}\le |\theta_{x,\,\tilde t+\tilde a_2}|\le 2^{-\tilde t}.$$
Taking $t_1:=\tilde t+\tilde a_2$ and $\frac r2:=2^{-\tilde t}$, leads to
$$ 2^{\tilde a_1-\tilde a_2-1}r\le |\theta_{x,\,t_1}|\le \frac r2.$$
This shows that
\begin{align}\label{e4.10}
\theta_{x,\,t_1}\subset B_{\rho_\Theta}(x,r)\  \ \mathrm {and} \ \ |\theta_{x,\,t_1}|\sim r.
\end{align}

Let $t\in \R$ such that $a_1 2^{-t} =r$.  Then for any ellipsoid $\theta=\theta_{y,\, s} \in\Theta$ with $|\theta|<r$ we have
\[
a_1 2^{-s} \le |\theta_{y,\, s}| < r = a_1 2^{-t} \le |\theta_{x,\, t}|.
\]
Hence, $t\le s$. In addition, if $x\in\theta_{y,\, s}$, then by Lemma \ref{l2.2}(ii)
$$
\theta_{y,\, s} \subset
\theta_{x,\, t_2}, \qquad\text{where }t_2:=t-\gamma.$$
Since $\theta \in \Theta$ with $x\in \theta$ and $|\theta|<r$ is arbitrary,
\eqref{e4.9} implies that
$$B_{\rho_\Theta}(x,r)\subset \theta_{x,\,t_2}\ \  \mathrm{and}\ \ |\theta_{x,\,t_2}|\sim r.$$
Combining this with \eqref{e4.10} completes the proof of Proposition \ref{p4.10}.
\end{proof}

It is often useful to use a non-symmetric variant of the quasi-distance $\rho_\Theta$ as in \eqref{e2.3}.

\begin{proposition}\label{p4.9}
Let $\Theta$ be a continuous ellipsoid cover. For any $x,\,y\in\rn$ define $\rho_1(x,y):=\inf_{y\in\theta_{x,\,t}\in\Theta}|\theta_{x,\,t}|$. Then $\rho_1(x,y) \sim \rho_\Theta(x,y)$ with equivalence constant independent of the choice of   $x,\,y\in\rn$.
\end{proposition}

\begin{proof}
Obviously, $\rho_\Theta(x,y)\le\rho_1(x,y)$. So it remains to prove that there exists a constant $C>0$ such that $\rho_1(x,y)\le C\rho_\Theta(x,y)$. Let $r= 2 \rho_\Theta(x,y)$. By Proposition \ref{p4.10}, there exist two ellipsoids $\theta_{x,\,t_1}$ and $\theta_{x,\,t_2}$ such that $|\theta_{x,\,t_1}|\sim|\theta_{x,\,t_2}|\sim r$ and
$$\theta_{x,\,t_1}\subset B_{\rho_\Theta}(x,r) \subset\theta_{x,\,t_2}.$$
Since $y\in B_{\rho_\Theta}(x,r) \subset\theta_{x,\,t_2}$, by the definition of $\rho_1$, we conclude that
$$\rho_1(x,y)\le|\theta_{x,\,t_2}|\sim r= 2 \rho_\Theta(x,y),$$
which completes the proof of Proposition \ref{p4.9}.
\end{proof}

Finally, we will need the following useful lemma.

\begin{lemma}\label{l4.12}
Let $z\in\mathbb{R}^n$, $t\in\mathbb R$, and $y\in\theta_{z,\,t}$. Then, for any $k\in \nn$ we have
\begin{equation}\label{l40}
\theta_{z,\,t-(k+1)\gamma}\setminus\theta_{z,\,t-k\gamma} \subset \theta_{y,\, t-(k+2) \gamma}\setminus\theta_{y,\,t-(k-1) \gamma},
\end{equation}
where $\gamma$ is as in Lemma \ref{l2.2}(ii).

In particular, if $x\in \theta_{z,\,t-(k+1)\gamma}\setminus\theta_{z,\,t-k\gamma}$ for some $k\in \nn$, then
\begin{equation}\label{e4.16}
\rho_\Theta(x,\,z)\sim\rho_{\Theta}(x,\,y) \sim 2^{-t+k\gamma}.
\end{equation}
\end{lemma}

\begin{proof}
Since
\[
y \in  \theta_{z,\,t-(k+1)\gamma } \cap \theta_{y,\,t-(k+1)\gamma} \ne \emptyset,
\]
by Lemma \ref{l2.2}(ii) we have $\theta_{z,\,t-(k+1)\gamma } \subset \theta_{y,\,t-(k+2)\gamma }$. Likewise, since
\[
y \in  \theta_{z,\,t-(k-1)\gamma } \cap \theta_{y,\,t-(k-1)\gamma} \ne \emptyset,
\]
Lemma \ref{l2.2}(ii) yields $\theta_{y,\,t-(k-1) \gamma } \subset \theta_{z,\,t-k\gamma }$. Therefore, by taking complements we deduce \eqref{l40}.

Finally, let $x\in \theta_{z,\,t-(k+1)\gamma}\setminus\theta_{z,\,t-k\gamma}$ for some $k\in \nn$.
By Proposition \ref{p4.9} and \eqref{l40} we have
$$\rho_\Theta(x,z)=\rho_\Theta(z,x)\sim\rho_1(z,x)\sim2^{-t+k\gamma}\sim \rho_1(y,x)\sim\rho_{\Theta}(y,x)
=\rho_{\Theta}(x,y).$$
This yields \eqref{e4.16}.
\end{proof}

%%%%%%%%%%%%%%%%%%%%%%%%%%%%%%%%%%%%%%%%%%%%%%%%%

%%%%%%%%%%%%%%%%%  Section 3 %%%%%%%%%%%%%%%%%%%%

%%%%%%%%%%%%%%%%%%%%%%%%%%%%%%%%%%%%%%%%%%%%%%%%%

\section{Variable Anisotropic Hardy Spaces}\label{s3}
In this section we recall the definition and properties of Hardy spaces with pointwise variable anisotropy which were originally introduced by Dekel, Petrushev, and Weissblat \cite{dpw}.

Let $\Theta$ be a continuous ellipsoid cover. For any locally integrable function $f$ on $\rn$, the {\it Hardy-Littlewood\
maximal\ operators\ $M_{B_{\rho_\Theta}}$\ and\ $M_\Theta$} are defined, respectively, to be
\begin{equation}\label{e3.1}
M_{B_{\rho_\Theta}} f(x):=\sup_{r>0}\frac1{|B_{\rho_\Theta}(x,r)|}\int_{B_{\rho_\Theta}(x,\,r)}|f(y)|dy
\end{equation}
and
\begin{equation}\label{e3.2}
M_{\Theta}f(x):=\sup_{t\in\mathbb{R}}\frac1{|\theta_{x,\,t}|}\int_{\theta_{x,\,t}}|f(y)|dy,
\end{equation}
where $B_{\rho_\Theta}(x,r)$ is as in \eqref{e3.3}.
By Proposition \ref{p4.10}, see \cite[Lemma 3.2]{dpw}, these two maximal functions are pointwise equivalent
\begin{equation}\label{l32}
M_{B_{\rho_\Theta}}f(x)\sim M_\Theta f(x)
\qquad\text{for all }f\in L^1_{\loc},  x\in\rn.
\end{equation}

\begin{definition}
Let $N,\wz{N}\in\nn_0$ with $N\le\wz{N}$. For a function $\varphi\in C^N$, let
$$\|\varphi\|_{N,\,\wz{N}}:=\max_{|\alpha|\le N} \sup_{y\in\rn}(1+|y|)^{\wz{N}}|\partial^\alpha\varphi(y)|. $$
Define
$$\mathcal{S}_{N,\,\wz{N}}:=\{\varphi\in\mathcal{S}:\,\|\varphi\|_{N,\,\wz{N}}\le 1\}.$$
\end{definition}

For each $x\in\rn$, $t\in\mathbb{R}$ and $\theta_{x,\,t}=M_{x,\,t}(\mathbb{B}^n)+x\in\Theta$, denote
$$\varphi_{x,\,t}(y):=|\det (M^{-1}_{x,t})| \varphi(M^{-1}_{x,\,t}y).$$

\begin{definition}\label{d3.2}
Let $f\in\mathcal{S'}$, $\varphi\in\mathcal{S}$, and $r\in(0,\fz)$. The {\it nontangential maximal function with aperture $r$} of $f$ is defined by
$$M_\varphi^r f(x):=\sup_{t\in\mathbb{R}}\,\sup_{y\in{r\,\theta_{x,\,t}}}|f\ast\varphi_{x,\,t}(y)|
\qquad {\rm for \ all} \ x\in\rn.$$
For any $N,\wz{N}\in\nn_0$ with $N\le\wz{N}$, the { \it nontangential grand maximal function with aperture $r$} of $f$ is defined by
$$M_{N,\,\wz{N}}^r f(x):=\sup_{\varphi\in\mathcal{S}_{N,\,\wz{N}}}M_\varphi^r f(x)\ \ \ {\rm for \ all} \ x\in\rn. $$
When aperture $r=1$, we obtain the {\it nontangential maximal function $M_\vz f$} and the {\it nontangential  grand maximal function} $M_{N,\,\wz N} f$ of $f$, respectively.
\end{definition}

\begin{definition}
Let $f\in\mathcal{S'}$ and $\varphi\in\mathcal{S}$. The $\it{radial\ maximal\ function}$ of $f$ is defined by
$$M^\circ_\varphi f(x):=\sup_{t\in\mathbb{R}}|f\ast\varphi_{x,\,t}(x)|\ \ \  {\rm for\ all}\ x\in\rn.$$
For any $N,\,\wz{N}\in\nn_0$ with $N\le\wz{N}$, the {\it radial grand maximal function} of $f$ is defined by
$$M^\circ_{N,\,\wz{N}}f(x):=\sup_{\varphi \,\in\mathcal{S}_{N,\,\wz{N}}}M^\circ_\varphi f(x)\ \ \ {\rm for\ all}\ x\in\rn.$$
\end{definition}

Next, we give the pointwise equivalence of the nontangential grand maximal function $M^r_{N,\,\wz{N}} f$ and the radial grand maximal function $M^\circ_{N,\,\wz{N}}f$ of $f$, the proof of which is motivated by \cite[p.17, Proposition 3.10]{b}.
\begin{theorem}\label{t3.4}
For any $N,\wz{N}\in\nn_0$ with $N\le\wz{N}$, there exists a constant $C:=C(N,\wz{N},r)$ such that for all $f\in\mathcal{S'}$,
$$M^\circ_{N,\,\wz{N}}f(x)\le M_{N,\,\wz{N}}^r f(x)\le C M^\circ_{N,\,\wz{N}}f(x) \ \ \ {\rm for\ a. e.} \ x\in\rn.$$
\end{theorem}

\begin{proof}
The first inequality is obvious. To show the second inequality, note that
\begin{align}\label{e3.4}
M_{N,\,\wz{N}}^r f(x)
&=\sup\{|f\ast\varphi_{x,\,t}(x+r M_{x,\,t}y)|:\,y\in\mathbb{B}^n,\,
t\in\mathbb{R},\,\varphi\in\mathcal{S}_{N,\,\wz{N}}\}\\
&=\sup\{|f\ast\phi_{x,\,t}(x)|:\,\phi(z):=\varphi(z+r y),\ y\in\mathbb{B}^n,\,t\in\mathbb{R},\,\varphi\in\mathcal{S}_{N,\,\wz{N}}\}\nonumber\\
&=\sup\{M^\circ_\phi f(x):\,\phi(z)=\varphi(z+r y),\ y\in\mathbb{B}^n,\,t\in\mathbb{R},\,\varphi\in\mathcal{S}_{N,\,\wz{N}}\}.\nonumber
\end{align}
For $\phi(z)=\varphi(z+r y)$ with $y\in\mathbb{B}^n$, we have
\begin{align}\label{e3.5}
\|\phi\|_{N,\,\wz{N}}&=\sup_{|\alpha|\le N}\sup_{x\in\rn}(1+|x|)^{\wz{N}}\lf|\partial^\alpha\varphi(x+r y)\r|\\
&=\sup_{|\alpha|\le N}\sup_{x\in\rn}(1+|x-r y|)^{\wz{N}}|\partial^\alpha\varphi(x)|\nonumber\\
&\le(1+r)^{\wz{N}}\sup_{|\alpha|\le N}\sup_{x\in\rn}(1+|x|)^{\wz{N}}|\partial^\alpha\varphi(x)|
=(1+r)^{\wz{N}}\|\varphi\|_{N,\,\wz{N}}.\nonumber
\end{align}
Combining \eqref{e3.4} and \eqref{e3.5}, we have
$$M_{N,\,\wz{N}}^r f(x)\le\sup\{M^\circ_\phi f(x):\, \phi\in\mathcal{S},\|\phi\|_{N,\,\wz{N}}\le(1+r)^{\wz{N}}\}
\le(1+r)^{\wz{N}}M^\circ_{N,\,\wz{N}}f(x),$$
which is desired and hence completes the proof of Theorem \ref{t3.4}.
\end{proof}

Hence, from Theorem \ref{t3.4} and a proof similar to that of \cite[Proposition 2.11(i)]{blyz}, we deduce that, for any $f\in\cs'\cap L^1_\loc$,
\begin{equation}\label{e3.6}
|f(x)|\le M^\circ_{N,\,\wz{N}}f(x)\le M_{N,\,\wz{N}}^r f(x)\le C M^\circ_{N,\,\wz{N}}f(x) \ \ \ {\rm for} \ x\in\rn.
\end{equation}

Let $\Theta$ be a continuous ellipsoid cover of $\rn$ with parameters $\mathbf{p}(\Theta)=\{a_1,\dots,a_6\}$
and let $0<p\le 1$. We define $N_p=N_p(\Theta)$ as the minimal integer satisfying
\begin{equation}\label{e3.7}
N_p(\Theta)>\frac{\max(1,a_4)n+1}{a_6p},
\end{equation}
and then $\wz{N}_p=\wz{N}_p(\Theta)$ as the minimal integer satisfying
\begin{equation*}
\wz{N}_p(\Theta)>\frac{a_4N_p(\Theta)+1}{a_6}.
\end{equation*}

\begin{definition}\label{d3.5}
Let $\Theta$ be a continuous ellipsoid cover, $0<p\le 1$ and $M^\circ:=M^\circ_{N_p,\,\wz{N}_p}$. The {\it variable anisotropic Hardy space} is defined as
$$H^p(\Theta):=\{f\in\mathcal{S'}:\,M^\circ f\in L^p\}$$
with the quasi-norm $\|f\|_{H^p(\Theta)}:=\|M^\circ f\|_p$.
\end{definition}

\begin{lemma}\label{l5.1}
Let $\Theta$ be a continuous ellipsoid cover and $0<p\le 1$.
\begin{enumerate}
\item[\rm{(i)}] The inclusion $H^p(\Theta)
\hookrightarrow\mathcal{S'}$ is continuous;
\item[\rm{(ii)}] $H^p(\Theta)$ is complete.
\end{enumerate}
\end{lemma}

\begin{proof}
To prove (i), for any $\phi\in\cs$, by \cite[Formula (5.8)]{dw}, we have
$$|\langle f,\phi\rangle|^p\le C\int_{\theta_{0,\,0}}(M^\circ f(x))^pdx\le C\|f\|^p_{H^p(\Theta)},$$
which implies that the inclusion $H^p(\Theta)
\hookrightarrow\mathcal{S'}$ is continuous. Mimicking the proof of \cite[Proposition 3.12]{b}, we can show (ii).
\end{proof}

As in the classical case, the anisotropic Hardy spaces can be characterized and then investigated through atomic decompositions. The following Definitions \ref{d3.6} and \ref{d3.7} come from \cite[Definitions 4.1 and 4.2]{dpw}, respectively.

\begin{lemma}\label{l4.1}
Let $\Theta$ be a pointwise continuous ellipsoid cover and $\lambda>0$. Then,
\[
\Omega=\{x:\,M^\circ f(x)>\lambda\} \text{ is open.}
\]
\end{lemma}

\begin{proof}
Let $f\in\cs'$ and $\varphi\in \cs_{N,\,{\wz{N}}}$. It suffices to show that for any $t\in\mathbb{R}$,
\begin{equation}\label{cla} \R^n \ni x\mapsto f\ast\varphi_{x,t}=f(\tau_x\wz{\varphi}_{x,t})\ {\rm is\ a \ countinous \ function},
\end{equation}
where $\tau_x \varphi(z):=\varphi(z-x)$ and $\wz{\varphi}(x):=\varphi(-x)$ with $z\in\rn$. Indeed, for any $x\in \Omega$, there exist $\varphi\in\mathcal{S}_{N,\,\wz{N}}$ and $t_0\in \mathbb{R}$ such that
$$\lf|f\ast\varphi_{x,t_0}(x)\r|>\lambda.$$
By \eqref{cla} we deduce that for $x'$ in a sufficiently small neighborhood of $x$, $|f\ast\varphi_{x',t_0}(x')|>\lambda$. This implies that $x'\in \Omega$ and hence $\Omega$ is open.

To prove \eqref{cla} we need to use the following fact. For any $\varphi \in \mathcal S$ and $n\times n$ matrix $M$ define $D_M\varphi (z): = \varphi(Mz)$, $z\in \R^n$. Then for a fixed $\varphi \in \mathcal S$, the  mapping
\begin{equation}\label{cla2}
\R^n \times GL_n(\R) \ni (x,M) \mapsto \tau_x D_M \varphi \in \mathcal S
\end{equation}
is continuous. To show this, it suffices to estimate the semi-norms $||\cdot||_{N,\wz N}$ defining the Schwartz class $\mathcal S$. Precomposing the mapping \eqref{cla2} with continuous mapping
\[
\R^n \ni x \mapsto (x, -(M_{x,t})^{-1}) \in \R^n \times GL_n(\R)
\]
yields \eqref{cla}. Here, we also used the fact that the inverse map $M \mapsto M^{-1}$ is continuous on $GL_n(\R)$.
\end{proof}

\begin{definition}\label{d3.6}
For a continuous ellipsoid cover $\Theta$, we say that $(p,q,l)$ is {\it admissible} if $0<p\le 1\le q\le\fz$, $p<q$ and $l\in\nn_0$, such that $l\ge N_p(\Theta)$ with $N_p(\Theta)$ as in \eqref{e3.7}. A {\it $(p,q,l)$-atom} is a function $a:\,\rn\rightarrow\mathbb{R}$ such that
\begin{enumerate}
\item[(i)]$\supp a\subset\theta_{x,\,t}$ for some $\theta_{x,\,t}\in\Theta$, where $x\in\rn$ and $t\in\mathbb{R}$;
\item[(ii)] $\|a\|_q\le|\theta_{x,\,t}|^{\frac1q-\frac1p}$;
\item[(iii)] $\int_\rn a(y)y^\alpha dy=0$ for all $\alpha\in\nn_0^n$ such that $|\alpha|\le l$.
\end{enumerate}
\end{definition}

\begin{definition}\label{d3.7}
Let $\Theta$ be a continuous ellipsoid cover and $(p,q,l)$ an admissible triple as in Definition \ref{d3.6}. The {\it atomic Hardy space} $H^p_{q,\,l}(\Theta)$ associated with $\Theta$ is defined to be the set of all tempered distribution $f\in\mathcal{S'}$ of the form $f=\sum^\fz_{i=1}\lambda_i a_i$, where the series converges in $\cs'$, $\{\lambda_i\}_i\subset\ccc$, $\sum^\fz_{i=1}|\lambda_i|^p<\fz$, and $\{a_i\}_i$ are $(p,q,l)$-atoms.  Moreover, the quasi-norm of $f\in H^p_{q,\,l}(\Theta)$ is defined by
$$\|f\|_{H^p_{q,\,l}(\Theta)}:=\inf \lf [\lf(\sum_i|\lambda_i|^p \r)^{\frac1p}\r],$$
where the infimum is taken over all admissible decompositions of $f$ as above.
\end{definition}

The following theorem due to Dekel, Petrushev, and Weissblat \cite[Sections 4.1 and 4.3]{dpw} shows the atomic characterization of $H^p(\Theta)$.

\begin{theorem}\label{l3.8}
Let $\Theta$ be a continuous ellipsoid cover, $0<p\le 1\le q\le\fz$, $p<q$ and $N_p(\Theta)\le l\in\nn_0$ with $N_p(\Theta)$ as in \eqref{e3.7}. Then $H^p(\Theta)=H^p_{q,\,l}(\Theta)$ with equivalent quasi-norms.
\end{theorem}

%%%%%%%%%%%%%%%%%%%%%%%%%%%%%%%%%%%%%%%%%%%%%%%%%

%%%%%%%%%%%%%%%%%  Section 4 %%%%%%%%%%%%%%%%%%%%

%%%%%%%%%%%%%%%%%%%%%%%%%%%%%%%%%%%%%%%%%%%%%%%%%

\section{Atomic decomposition of anisotropic Hardy spaces}\label{s4}

In this section we show technical improvements in the Calder\'on-Zygmund decomposition and the atomic decomposition of variable anisotropic Hardy space $H^p(\Theta)$, which were originally established by Dekel, Petrushev, and Weissblat in the form of Theorem \ref{l3.8}. While these are incremental improvements of results in \cite{dpw}, they play a crucial role in the proofs of Theorems \ref{t4.17} and \ref{t4.16} in Section \ref{s6}.

We recall the Calder\'{o}n-Zygmund decomposition established in \cite{dpw}.
Throughout this section we fix an ellipsoid cover $\Theta$ and we consider a tempered distribution $f$ such that for every $\lambda>0$, $|\{x:\,M^\circ f(x)>\lambda\}|<\fz$, where $M^\circ$ is the grand maximal function as in Definition \ref{d3.5}. We shall assume that $\Theta$ is pointwise continuous, that is \eqref{e2.0} holds.
This condition is implicitly used in \cite{dpw} to guarantee that the set
\begin{equation}\label{omega}
\Omega:=\{x:\,M^\circ f(x)>\lambda\}
\end{equation}
is open, see Lemma \ref{l4.1}. However, in light of Theorem \ref{t2.1} this assumption can be later removed, for example in the statement of Theorem \ref{l5.9}, since the Hardy spaces $H^p(\Theta)$ and $H^p(\Xi)$ corresponding to equivalent ellipsoid covers $\Theta$ and $\Xi$ are the same with equivalent quasi-norms.

By covering arguments \cite[Section 4.2]{dpw}, there exist sequences $\{x_i\}_{i\in\nn_0}\subset\Omega$ and $\{t_i\}_{i\in\nn_0}$, such that
\begin{equation}\label{e5.1}
\Omega=\bigcup_{i\in\nn_0}\theta_{x_i,\,t_i},
\end{equation}
\begin{equation}
\theta_{x_i,\,t_i+\gamma}\cap\theta_{x_j,\,t_j+\gamma}=\emptyset \qquad \forall\ i\neq j,
\end{equation}
\begin{equation}\label{e5.3}
\theta_{x_i,\,t_i-J-2\gamma}\cap\Omega^\complement=\emptyset \qquad \forall\ i\in\nn_0,
\end{equation}
\begin{equation}
\theta_{x_i,\,t_i-J-2\gamma-1}\cap\Omega^\complement\neq\emptyset \qquad \forall\ i\in\nn_0,
\end{equation}
where $J$ and $\gamma$ are as in Lemma \ref{l2.2}. Moreover, there exists a constant $L>0$ such that
\begin{equation}\label{e5.5}
\sharp\{j\in\nn_0:\,\theta_{x_j,\,t_j-J-\gamma}\cap\theta_{x_i,\,t_i-J-\gamma}\neq\emptyset\}\le L \qquad \forall\ i\in\nn_0,
\end{equation}
where $\sharp E$ denotes the cardinality of a set $E$.

Fix $\phi\in C^\fz$ such that $\supp\phi\subset 2\mathbb{B}^n$, $0\le\phi\le 1$ and $\phi\equiv 1$ on $\mathbb{B}^n$. For every $i\in\nn_0$, define $\wz{\phi_i}:=\phi(M^{-1}_{x_i,\,t_i}(x-x_i))$. Obviously, $\wz{\phi_i}\equiv 1$ on $\theta_{x_i,\,t_i}$. By Lemma \ref{l2.2}(i), we have $\supp{\wz \phi_i}\subset x_i+2M_{x_i,\,t_i}(\mathbb{B}^n) \subset\theta_{x_i,\,t_i-J}$. For
every $i\in\nn_0$, define
\begin{equation}\label{e5.6}
\phi_i(x):=
\lf\{\begin{array}{ll}\frac{\wz{\phi_i}(x)}{\sum_j\wz{\phi}_j(x)},\quad & {\rm if}\ x\in\Omega,\\
0,& {\rm if}\ x\notin\Omega.
\end{array}\r.
\end{equation}
Observe that $\phi_i$ is well defined since by \eqref{e5.1} and \eqref{e5.5}, $1\le\sum_i\wz{\phi_i}(x)\le L$ for every $x\in\Omega$. Also $\phi_i\in C^\fz$ and $\supp\phi_i\subset\theta_{x_i,\,t_i-J}$. By \eqref{e5.1} and \eqref{e5.6}, we have $\sum_i\phi_i(x)=\textbf{1}_\Omega(x)$, which implies that the family $\{\phi_i\}_{i\in\nn_0}$ forms a smooth partition of unitary subordinate to the cover of $\Omega$ by the ellipsoids $\{\theta_{x_i,\,t_i-J}\}_{i\in\nn_0}$.

Let $\mathcal{P}_l$ denote the space of polynomials of $n$ variables with degree $\le l$, where $N_p(\Theta)\le l$, see \eqref{e3.7}. For each $i\in\nn_0$ we introduce an Hilbert space structure on the space $\mathcal{P}_l$ by setting
\begin{equation}\label{e5.6a}
\langle P,Q\rangle_i:=\frac1{\int\phi_i}\int_\rn P(x)Q(x)\phi_i(x)dx\ \ \ {\rm for\ any}\ P,\,Q\in\mathcal{P}_l.
\end{equation}
The distribution $f\in\cs'$ induces a linear functional on $\mathcal{P}_l$ given by
$$
\mathcal{P}_l \ni Q\mapsto \langle f,Q\rangle_i.$$
By the Riesz Lemma it is represented by a unique polynomial $P_i\in\mathcal{P}_l$ such that
\begin{equation}\label{e5.7}
\langle f,Q\rangle_i=\langle P_i,Q\rangle_i \ \ \ {\rm for\ any}\ Q\in\mathcal{P}_l.
\end{equation}

\begin{definition}\label{d5.2}
For every $i\in\nn_0$, define the locally ``bad part'' $b_i:=(f-P_i)\phi_i$ and the ``good part'' $g:=f-\sum_ib_i$.
The representation $f=g+\sum_ib_i$, where $g$ and $b_i$ as above, is a Calder\'{o}n-Zygmund decomposition of degree $l$ and height $\lambda$ associated with $M^\circ$.
\end{definition}

We will use the following three results, which are \cite[Lemma 4.8, Lemma 4.11(ii), and Lemma 4.13]{dpw}, respectively. In particular, Lemma \ref{l5.4} guarantees convergence of $\sum_i b_i$.

\begin{lemma}\label{l5.3}
There exists a positive constant $C$ such that
$$\sup_{y\in\rn}|P_i(y)\phi_i(y)|\le C\lambda,$$
where $\phi_i$ and $P_i$ are defined in \eqref{e5.6} and \eqref{e5.7}, respectively.
\end{lemma}

\begin{lemma}\label{l5.4}
Suppose $f\in H^p(\Theta)$, $0<p\le 1$. Then the series $\sum_ib_i$ converges in $H^p(\Theta)$ and there exists a positive constant $C$, independent of $f$ and $i\in\nn_0$, such that
$$\int_\rn \lf[M^\circ\lf(\sum_ib_i\r)(x)\r]^pdx\le C\int_\Omega (M^\circ f(x))^pdx.$$
\end{lemma}

\begin{lemma}\label{l5.5}
Suppose $\sum_ib_i$ converges in $\cs'$. Then there exists a positive constant $C$, independent of $f\in\cs'$ and $\lambda>0$, such that
$$M^\circ g(x)\le C\lambda\sum_i\nu^{-k_i(x)}+M^\circ f(x)\textbf{1}_{\Omega^\complement}(x),$$
where $\nu:=2^{\,a_6JN}$ and
$$k_i(x):=
\lf\{\begin{array}{ll}
k,\ & x\in\theta_{x_i,\,t_i-J(k+2)}\setminus\theta_{x_i,\,t_i-J(k+1)}\
{\rm \ for\ some \ }\ k\in\nn_0,\\
0,& x\in\theta_{x_i,\,t_i-J}.
\end{array}\r.$$
\end{lemma}

The following two lemmas are extensions of \cite[Lemmas 4.12 and 4.14]{dpw} from the setting of $L^1$ to $L^q$ spaces, $1< q<\infty$. At the same time, these results are extensions of \cite[Lemmas 4.8 and 4.10]{blyz} to the variable anisotropic setting (albeit without weights).

\begin{lemma}\label{l5.6}
If $f\in L^q$ with $1\le q<\fz$. Then $\sum_ib_i$ converges in $L^q$.
Moreover, there exists a positive constant $C$, independent of $f$ and $i$, such that $\|\sum_i|b_i|\|_{q}\le C\|f\|_{q}.$
\end{lemma}

\begin{proof}
When $q=1$, this result was shown in \cite[Lemma 4.12]{dpw}. Hence, we only need to consider the case $1<q<\fz$.

When $1<q<\fz$, from $b_i=(f-P_i)\phi_i$ and Lemma \ref{l5.3}, it follows that
\begin{align}\label{e5.8}
\int_\rn|b_i(x)|^qdx
&=\int_\rn|(f(x)-P_i(x))\phi_i(x)|^qdx\\
&\le 2^{q-1} \lf( \int_{\theta_{x_i,\,t_i-J}}|f(x)\phi_i(x)|^qdx+\int_{\theta_{x_i,\,t_i-J}}|P_i(x)\phi_i(x)|^qdx \r) \nonumber\\
&\le C\lf(\int_{\theta_{x_i,\,t_i-J}}|f(x)|^qdx+\lambda^q|\theta_{x_i,\,t_i-J}|\r).\nonumber
\end{align}
For the set $\Omega$ as in \eqref{e5.1}, by Lemma \ref{l2.2}, \eqref{e5.1}, and \eqref{e5.3} we have
\begin{equation}\label{e5.9}
\Omega=\bigcup_{i\in\nn_0}\theta_{x_i,\,t_i-J}.
\end{equation}
Therefore, by \eqref{e5.5}, \eqref{e5.8}, \eqref{e5.9}, and the $L^q$ boundedness of $M^\circ$ (see \cite[Theorem 3.8]{dpw}), we have
\begin{align*}
\sum_i\int_\rn|b_i(x)|^qdx
&\le C\lf(\sum_i\lf(\int_{\theta_{x_i,\,t_i-J}}|f(x)|^qdx+
\lambda^q|\theta_{x_i,\,t_i-J}|\r)\r)\nonumber\\
&\le CL\lf(\int_\Omega|f(x)|^qdx+\lambda^q|\Omega|\r)\\
&\le CL\lf(\int_\Omega|f(x)|^qdx+\int_\rn (M^\circ f(x))^qdx\r)\nonumber\\
&\ls \int_\rn|f(x)|^qdx\nonumber.
\end{align*}
Since $\supp b_i\subset\supp\phi_i\subset\theta_{x_i,\,t_i-J}$, by H\"{o}lder's inequality and \eqref{e5.5}, we deduce that
\begin{align*}
\lf(\int_\rn\lf(\sum_i\lf|b_i(x)\r|\r)^qdx\r)^{\frac1{q}}
&\le\lf(\int_\rn\lf(\sum_i\lf|b_i(x)\textbf{1}_{\theta_{x_i,\,t_i-J}}(x)\r|\r)^qdx\r)^{\frac1{q}}\\
&\le\lf\{\int_\rn\lf[\lf(\sum_i|b_i(x)|^q\r)^{\frac1q}\lf(\sum_i|\textbf{1}_{\theta_{x_i,\,t_i-J}}(x)|^{q'}\r)^{\frac 1{q'}}\r]^qdx\r\}^{\frac1{q}}\\
&\le L^{1-\frac1q}\lf(\sum_i\int_\rn|b_i(x)|^qdx\r)^{\frac1{q}}\\
&\le C\lf(\int_\rn|f(x)|^qdx\r)^{\frac1{q}},
\end{align*}
where $1/q+1/q'=1$.
\end{proof}

\begin{lemma}\label{l5.7}
If $M^\circ f\in L^p$ with $0<p\le 1$, then $M^\circ g\in \bigcap_{1\le q<\fz}L^q$. Moreover, there exists a positive constant $C_1$, independent of $f$ and $\lambda$, such that, when $1\le q<\fz$,
$$\int_\rn (M^\circ g(x))^qdx\le C_1\lambda^{q-p}\int_\rn (M^\circ f(x))^pdx.$$

If $f\in L^q$ with $1\le q<\fz$, then $g\in L^\fz$ and there exists a positive
constant $C_2$, independent of $f$ and $\lambda$, such that
\begin{equation*}
\|g\|_{\fz}\le C_2\lambda.
\end{equation*}
\end{lemma}

\begin{proof}
When $q=1$, this result was shown in \cite[Lemma 4.14]{dpw}. Hence, we only need to consider the case $1<q<\fz$.

When $1<q<\fz$, by Lemma \ref{l5.5}, we obtain
\begin{align}\label{e5.10}
\int_\rn [M^\circ g(x)]^qdx
&\le C\lf[\lambda^q\int_\rn\lf(\sum_i\nu^{-k_i(x)}\r)^qdx+\int_{\Omega^\complement}(M^\circ f(x))^qdx\r],
\end{align}
where $k_i(x)$ is as in Lemma \ref{l5.5}. For any $x\in\theta_{x_i,\,t_i-J(k+2)}\setminus\theta_{x_i,\,t_i-J(k+1)}$ with $k\in\nn_0$, by \eqref{l32}, we have
\begin{align*}2^{-kJ}&\le C\frac1{|\theta_{x_i,\,t_i-J(k+2)}|}\int_{\theta_{x_i,\,t_i-J(k+2)}}
\textbf{1}_{\theta_{x_i,\,t_i}}(y)dy\\
&\le C M_{\Theta}(\textbf{1}_{\theta_{x_i,\,t_i}})(x)\sim M_{B_{\rho_\Theta}}(\textbf{1}_{\theta_{x_i,\,t_i}})(x),
\end{align*}
where $M_{B_{\rho_\Theta}}$ and $M_\Theta$ are as in \eqref{e3.1} and \eqref{e3.2}, respectively.

From the Fefferman-Stein boundedness of the vector-valued maximal function, see \cite[Section 6.6]{ggkk}, $M_{B_{\rho_\Theta}}$ is bounded on  $L^{a_6Nq}(\ell^{a_6N})$ space with $a_6N>1$. Hence, $\nu=2^{\,a_6JN}$, \eqref{e5.1}, and \eqref{e5.5}, yield
\begin{align*}
\int_\rn\lf(\sum_i\nu^{-k_i(x)}\r)^qdx&=\int_\rn\lf(\sum_i2^{-k_i(x)Ja_6N}\r)^qdx\\
&\ls \int_\rn\lf\{\lf[\sum_i\lf(M_{B_{\rho_\Theta}}(\textbf{1}_{\theta_{x_i,\,t_i}})(x)\r)^{a_6N}\r]^{1/(a_6N)}\r\}^{a_6Nq}dx\\
&\ls \int_\rn\lf[\sum_i(\textbf{1}_{\theta_{x_i,\,t_i}}(x))^{a_6N}\r]^qdx\\
&\ls \int_\Omega dx\sim|\Omega|.
\end{align*}
By \eqref{e5.10}, we further conclude that
\begin{align*}
\int_\rn (M^\circ g(x))^qdx
&\le C\lf[ \lambda^q|\Omega|+\int_{\Omega^\complement}(M^\circ f(x))^qdx\r]\\
&\le C\lf[ \lambda^{q-p}\int_\Omega (M^\circ f(x))^pdx+\lambda^{q-p}\int_{\Omega^\complement}(M^\circ f(x))^pdx\r]\\
&\le C\lambda^{q-p}\int_\rn (M^\circ f(x))^pdx.
\end{align*}
This implies that $M^\circ g\in \bigcap_{1\le q<\fz}L^q$.

Next suppose that $f\in L^q$, $1\le q<\fz$. By Lemma \ref{l5.6}, we deduce that $g$ and $b_i$, $i\in\nn_0$, are functions and $\sum_{i\in\nn_0}b_i$ converges in $L^q$. Thus, for a.e. $x\in\rn$,
$$g=f-\sum_ib_i=f\textbf{1}_{\Omega^\complement}+\sum_iP_i\phi_i.$$
By Lemma \ref{l5.3} and \eqref{e5.5}, for every $x\in\Omega$, we have $|g(x)|\le C\lambda$. Moreover, for a.e. $x\in\Omega^\complement$, by \eqref{e3.6}, we obtain that $|g(x)|=|f(x)|\le M^\circ f(x)\le \lambda$. Therefore $\|g\|_{\fz}\le C\lambda$.
\end{proof}

Motivated by \cite[Corollary 28]{lyy} and \cite[Corollary 4.11]{blyz}, we have the following lemma, which is an extension of \cite[Corollary 4.15]{dpw} from the setting $L^1$ to $L^q$ spaces, $ 1\le q<\fz$.
\begin{lemma}\label{l5.8}
For any $0<p\le 1$ and $1\le q<\fz$, the subspace $H^p(\Theta)\cap L^q$ is dense in $H^p(\Theta)$.
\end{lemma}

\begin{proof}
Let $f\in H^p(\Theta)$. For any $\lambda>0$, let $f=g^\lambda+\sum_ib^\lambda_i$ be the Calder\'{o}n-Zygmund decomposition of $f$ of degree $l\ge N_p(\Theta)$ and height $\lambda$ associated with $M^\circ$ as in Definition \ref{d5.2}.
By Lemma \ref{l5.4}, we know that
$$\lf\|\sum_ib^\lambda_i\r\|_{H^p(\Theta)}^p\le C\int_{\{x: M^\circ f(x)>\lambda\}}[M^\circ f(x)]^pdx\rightarrow 0\ \ {\rm as}\ \lambda\rightarrow\fz,$$
which implies $g^\lambda\rightarrow f$ in $H^p(\Theta)$. Moreover, by Lemma \ref{l5.7}, we have $M^\circ g^\lambda\in L^q$, $1\le q<\fz$. From this and \eqref{e3.6}, we deduce that $g^\lambda\in L^q$, $1\le q<\fz$. This finishes the proof of Lemma \ref{l5.8}.
\end{proof}

Following \cite[Section 4.3]{dpw}, for each $k\in\mathbb{Z}$, we consider the Calder\'{o}n-Zygmund decomposition of $f$ of degree $l\ge N_p(\Theta)$ at height $2^k$ associated with $M^\circ$,
\begin{equation*}
f=g^k+\sum_ib^k_i,
\end{equation*}
where
$$\Omega^k:=\{x:\,M^\circ f>2^k\},\qquad
 b^k_i:=(f-P^k_i)\phi^k_i,
 \qquad
 \theta^k_i:=\theta_{x^k_i,\,t^k_i}.$$
Here, sequences $\{x^k_i\}_{i\in\nn_0} \subset \Omega^k$ and $\{t^k_i\}_{i\in\nn_0} \subset \R$ satisfy \eqref{e5.1}-\eqref{e5.5} for $\Omega^k$, functions $\{\phi^k_i\}_{i\in\nn_0}$ are defined as in \eqref{e5.6}, and polynomials $\{P^k_i\}_{i\in\nn_0}$ are projections of $f$ onto $\mathcal{P}_l$ with respect to the inner product given by \eqref{e5.6a}.

Next, we define $P^{k+1}_{ij}$ as the orthogonal projection of $(f-P^{k+1}_j)\phi^k_i$ with respect to the inner product
$$\langle P,Q\rangle_j:=\frac1{\int\phi^{k+1}_j}\int_\rn P(x)Q(x)\phi^{k+1}_j(x)dx\ \ \ {\rm for\ all}\ P,\,Q\in\mathcal{P}_l.$$
That is, $P^{k+1}_{ij}$ is the unique polynomial in $\mathcal{P}_l$ such that
$$\int_\rn(f(y)-P^{k+1}_j(y))\phi^k_i(y)Q(y)\phi^{k+1}_j(y)dy=\int_\rn P^{k+1}_{ij}(y)Q(y)\phi^{k+1}_j(y)dy\ \ {\rm for\ all}\ Q\in\mathcal{P}_l.$$
In particular, if $\theta_{x^k_i,\,t^k_i-J}\cap\theta_{x^{k+1}_j,\,t^{k+1}_j-J} =\emptyset$, then $P^{k+1}_{ij}= 0$.

For each $k\in \Z$, define the index set
\begin{equation*}
I_k:=\{(i,j)\in\nn_0:\theta_{x^k_i,\,t^k_i-J}\cap
\theta_{x^{k+1}_j,\,t^{k+1}_j-J}\neq\emptyset\}.
\end{equation*}
We will need to employ two additional results \cite[Lemmas 4.16 and 4.17]{dpw}, respectively.

\begin{lemma}\label{l5.10} The following holds for any $k\in \Z$.
\begin{enumerate}
\item[(i)] For any $(i,j)\in I_k$ we have $\theta_{x^{k+1}_j,\, t^{k+1}_j -J} \subset
\theta_{x^{k}_i,\, t^{k}_i -J-3\gamma -1}$,
\item[(ii)]
There exists $L'>0$, which does not depend on $k$, such that$$
\sharp \{i\in\nn_0:\,
(i,j) \in I_k \} \le L' \qquad\text{for any }j\in \nn_0.$$
\end{enumerate}
\end{lemma}

\begin{lemma}\label{l5.11}
There exists a constant $C>0$, such that for every $i,\,j\in\nn_0$ and $k\in\mathbb{Z}$,
$$\sup_{x\in\rn}\lf|P^{k+1}_{ij}(x)
\phi^{k+1}_{j}(x)\r|\le C2^{k+1}.$$
Moreover, $P^{k+1}_{ij}= 0$ if $(i,j) \not \in I_k$.
\end{lemma}

Motivated by \cite[Proposition 4.10]{hlyy}, we have the following extension of Theorem \ref{l3.8}, which yields convergence of atomic decompositions in $L^q$ norm.

\begin{theorem}\label{l5.9}
Let $0<p\le 1$, $1<q<\fz$, and $l\ge N_p(\Theta)$. Then, for any $f\in L^q\cap H^p(\Theta)$,
there exist a sequence of $(p,\fz,l)$-atoms $\{a^k_i\}_{k\in\mathbb Z,\, i\in\nn_0}$, a sequence $\{\lambda^k_i\}_{k\in\mathbb Z,\,i\in\nn_0}\subset\mathbb{C}$,  and a positive constant $C$ independent of $f$ such that
\begin{equation}\label{ad1}
\sum_{k\in\mathbb{Z}}\sum_{i\in\nn_0}|\lambda^k_i|^p\le C\|f\|^p_{H^p(\Theta)}
\end{equation}
and
\begin{equation}\label{ad2}
f=\sum_{k\in\zz}\sum_{i\in\nn_0}\lambda^k_ia^k_i\qquad \text{converges  in }L^q.
\end{equation}
\end{theorem}

\begin{proof}
Let $f\in L^q\cap H^p(\Theta)$ with $1\le q<\infty$ and $0<p\le 1$. Following the proof of \cite[Theorem 4.19]{dpw} with \cite[Lemma 4.14]{dpw} being replaced by Lemma \ref{l5.7}, we obtain the same conclusion, an atomic decomposition of $f$, under the assumption $f\in L^q\cap H^p(\Theta)$ instead of $f\in L^1\cap H^p(\Theta)$. More precisely, define a sequence of functions $\{h^k_i\}_{k\in\zz,\,i\in\nn_0}$ by
\begin{equation*}
h^k_i:=f{\bf 1}_{(\Omega^{k+1})^\complement}\phi^k_i
-P^k_i\phi^k_i
+\sum_{j\in\nn_0}P^{k+1}_j\phi^{k+1}_j\phi^k_i
+\sum_{j\in\nn_0}P^{k+1}_{ij}\phi^{k+1}_j.
\end{equation*}
By Lemma \ref{l5.3} and Lemma \ref{l5.11} we conclude that
\begin{equation}\label{2k}
\|h^k_i\|_\fz\le C 2^k.
\end{equation}
Moreover, there exist a sequence of $(p,\fz,l)$-atoms  $\{a^k_i\}_{k\in\zz,\,i\in\nn_0}$, which are supported on ellipsoids $\{\theta_{x^k_i,\,t^k_i-J-3\gamma-1}\}_{k\in\zz,\,i\in\nn_0}$,  and a sequence $\{\lambda^k_i\}_{k\in\zz,\,i\in\nn_0}\subset\mathbb{C}$, such that \eqref{ad1} holds, $h^k_i =\lambda^k_i a^k_i$ for all $i,j\in \nn_0$, and
\begin{equation}\label{e5.11}
f=\sum_{k\in\mathbb{Z}}\sum_{i\in\nn_0}\lambda^k_i a^k_i = \sum_{k\in\mathbb{Z}}\sum_{i\in\nn_0}h^k_i\qquad\text{converges in }\cs'.
\end{equation}

It remains to prove that the atomic decomposition \eqref{e5.11} also converges in $L^q$ norm.
Since $\supp\phi^{k+1}_j\subset\theta_{x^{k+1}_j,\,t^{k+1}_j-J}$ and $\supp\phi^k_i\subset\theta_{x^k_i,\,t^k_i-J}$, by Lemma \ref{l5.11} we have
$$\supp h^k_i\subset\theta_{x^k_i,\,t^k_i-J}\cup\lf(\bigcup_{j\in\nn_0 ,\ (i,\,j)\in I_k}\theta_{x^{k+1}_j,\,t^{k+1}_j-J}\r).$$
Hence, by \eqref{e5.1}, \eqref{e5.3}, and \eqref{e5.5} applied at levels $k$ and $k+1$ and by Lemma \ref{l5.10},  it follows that
\begin{align*}
\sum_{i\in\nn_0}{\bf 1}_{\supp h^k_i}
&\le\sum_{i\in\nn_0}{\bf 1}_{\theta_{x^k_i,\,t^k_i-J}}
+\sum_{(i,j)\in I_k}{\bf 1}_{\theta_{x^{k+1}_j,\,t^{k+1}_j-J}}
\le L{\bf 1}_{\Omega^k}+\sum_{j\in\nn_0}\sum_{i\in \nn_0, (i,j)\in I_k}{\bf 1}_{\theta_{x^{k+1}_j,\,t^{k+1}_j-J}}\\
&\le L{\bf 1}_{\Omega^k}+LL'{\bf 1}_{\Omega^{k+1}}\le L(1+L'){\bf 1}_{\Omega^k}.
\end{align*}
In the last step we used $\Omega^{k+1}\subset\Omega^k$.
This together with \eqref{2k} implies that
\begin{equation}\label{e5.x1}
\sum_{i\in\nn_0}|h^k_i|\le  C L(1+L') 2^k{\bf 1}_{\Omega^k}.
\end{equation}

Since $f\in L^q\cap H^p(\Theta)$, for almost every $x\in\R^n$, there exists $k(x)\in\zz$ such that $2^{k(x)}<M^\circ f(x)\le 2^{k(x)+1}$. From this and \eqref{e5.x1}, we deduce that, for a.e. $x\in\rn$,
\begin{align*}
\sum_{k\in\zz}\sum_{i\in\nn_0}|h^k_i(x)|
&\le CL(1+L')\sum_{k\in(-\fz,k(x)]\cap\zz}2^k{\bf 1}_{\Omega^k}\\
&\le CL(1+L')2^{k(x)}{\bf 1}_{\Omega^{k(x)}}\sim M^\circ f(x).\nonumber
\end{align*}
Therefore, the series $\sum_{k\in\zz}\sum_{i\in\nn_0}h^k_i$ converges absolutely pointwise a.e. to some function $\wz{f} \in L^q$. By the Lebesgue dominated convergence theorem we deduce that $\wz{f}=\sum_{k\in\zz}\sum_{i\in\nn_0}h^k_i$ converges unconditionally in $L^q$. By \eqref{e5.11} we necessarily have $f=\wz{f}\in L^q$, which yields \eqref{ad2}.
\end{proof}

%%%%%%%%%%%%%%%%%%%%%%%%%%%%%%%%%%%%%%%%%%%%%%%%%

%%%%%%%%%%%%%%%%%  Section 5 %%%%%%%%%%%%%%%%%%%%

%%%%%%%%%%%%%%%%%%%%%%%%%%%%%%%%%%%%%%%%%%%%%%%%%

\section{Variable Anisotropic Singular Integral Operators}\label{s5}
In this section, we introduce the notion of variable anisotropic singular integral operators associated with a continuous ellipsoid cover $\Theta$ and show that such operators are bounded from $H^p(\Theta)$ to $L^p$ and from $H^p(\Theta)$ to itself for $0<p\le 1$.

Coifman and Weiss \cite[Chapter III.2]{cw71} have introduced the general notion of singular integral operators defined on arbitrary spaces of homogeneous type. By Proposition \ref{p2.4} a continuous ellipsoid cover $\Theta$ induces a quasi-distance $\rho_\Theta$ with respect to which $\R^n$ becomes a space of homogeneous type. This leads to the definition of singular integral operators associated with a continuous ellipsoid cover $\Theta$, which satisfy H\"ormander's condition, see \cite[Chapter III.2]{cw71} and \cite[Chapter I.5]{s}.

\begin{definition}\label{d4.1}
A locally square integrable function $K$ on $\Omega:=\{(x,y)\in\rn\times\rn:\,x\ne y\}$ is called a {\it variable anisotropic singular integral kernel} with respect to a continuous ellipsoid cover $\Theta$ if there exist two positive constants $c>1$ and $C$ such that
\begin{equation}\label{e4.1}
\int_{B_{\rho_\Theta}(y,\,cr)^\complement}|K(x,y)-K(x,y')|dx\le C,
\end{equation}
where $y'\in B_{\rho_\Theta}(y,r)$, $y\in\rn$ and $B_{\rho_\Theta}(y,r)$ is as in \eqref{e3.3}.

We say that $T$ is a {\it variable anisotropic singular integral operator} (VASIO) of order $0$ if $T: L^2 \to L^2$ is a bounded linear operator if there exists a kernel $K$ satisfying \eqref{e4.1} such that
$$Tf(x)=\int_\rn K(x,y)f(y)dy \qquad\text{for all } f\in C_c^\infty, \ x\not\in \supp (f).$$
\end{definition}

The fundamental theorem about singular integral operators, which holds on arbitrary spaces of homogeneous type, asserts that $T$ is also bounded from $L^1$ to weak-$L^1$. Then, the Marcinkiewicz interpolation theorem implies that $T$ is bounded from $L^q$ to $L^q$ for $1<q\le 2$. Thus, we have the following theorem, see \cite[Theorem III.2.4]{cw71} and \cite[Theorem I.3]{s}.

\begin{theorem}\label{t4.2}
Let $T$ be a {\rm VASIO} of order 0 and $1<q\le 2$. Then $T$ extends to a bounded linear operator $L^q \to L^q$.
\end{theorem}

\begin{remark}\label{r4.3}
To get boundedness for entire range of $1<q<\infty$, one needs to impose a symmetric variant of \eqref{e4.1} with variables $x$ and $y$ being interchanged, which by the duality yields the boundedness for $2<q<\infty$.
\end{remark}

Since we are interested in the boundedness of singular integral operators on the Hardy spaces $H^p(\Theta)$, $0<p\le 1$, we need to impose smoothness hypothesis on the kernel $K$, which is much stricter than that given by Definition \ref{d4.1}. To this end we shall extend the definition of Calder\'on-Zygmund operators in anisotropic setting which was given in \cite[Definition 9.2]{b}.

\begin{definition}\label{d4.4}
Let $s\in\nn_0$ and let $T$ be a VASIO as in Definition \ref{d4.1} with kernel $K(x,y)$ in the class $C^s$ as a function of $y$. Then we say that $T$ is a {\it VASIO\ of\ order\ $s$} if there exists a constant $C>0$ such that for any $(x,y) \in\Omega$ and for any multi-index $|\alpha| \le s$ we have
\begin{equation}\label{e4.2}
\lf|\partial^\alpha_y[K(\cdot,M_{y,\,m}\cdot)](x,M^{-1}_{y,\,m}y)\r|\le C/\rho_\Theta(x,y) \qquad\text{where } m:= -  \log_{2} \rho_\Theta(x,y).
\end{equation}
More precisely, the left hand side of \eqref{e4.2} means $|\partial^\alpha_y \wz{K}(x,M^{-1}_{y,\,m}\,y)|$, where $\wz{K}(x,y):=K(x,M_{y,\,m}\,y)$. The smallest constant $C$ satisfying \eqref{e4.2} is called a Calder\'on-Zygmund norm of $T$, which is denoted by $\|T\|_{(s)}$.
\end{definition}

Next we will show that Definition \ref{d4.4} is an extension of the class of Calder\'{o}n-Zygmund operators associated with expansive dilations, which was introduced in \cite[Definition 9.2]{b}. For this we need the following estimate on higher order derivatives under linear change of variables, which is stated implicitly in \cite[p.25]{b}.

\begin{lemma}\label{l4.7}
Let $h$ be a function in the class $C^s$, $s\in \nn_0$, defined on an open subset $U \subset \R^n$. Let $M: \R^n \to \R^n$ be a linear invertible map. Let $\tilde h$ be dilation of $h$ by $M$ defined by $\tilde h(x)=h(Mx)$ for $x\in M^{-1}U$. Then, there exists a constant $C=C(s,n)>0$ depending only on $s$ and the dimension $n$ such that
\begin{equation}\label{e4.8}
|\partial^\alpha [h(M \cdot )](M^{-1}x)| = |\partial ^\alpha \tilde h(M^{-1}x)| \le C \|M\|^{|\alpha|} \sup_{|\beta|=|\alpha|} |\partial^\beta h(x)|
\qquad\text{for }x\in U, \ |\alpha|\le s.
\end{equation}
\end{lemma}

\begin{proof}
For any $k=1,\ldots, s$, and $x\in U$, let $\mathfrak D^k h(x)$ be the total derivative of $h$ at $x$ of order $k$, which is a symmetric multilinear functional $\mathfrak D^k h(x): \R^{nk} = \R^n \times \dots \times \R^n \to \R$. The norm of multilinear functional is given by
\[
\|\mathfrak D^k h(x)\|:= \sup \{|\mathfrak D^k h(x)(v_1,\ldots,v_k)|:  v_i \in \R^n, |v_i|=1, i=1,\ldots, k\}.
\]
For any multi-index $\alpha=(\alpha_1,\ldots,\alpha_n) \in \nn_0^n$, $|\alpha|=k$, let $\sigma: \{1,\ldots, k\} \to \{1,\ldots, n\}$ be the mapping that takes each value $j=1,\ldots, n$ exactly $\alpha_j$ times. Then, the partial and total derivates satisfy the relationship
\[
\partial^\alpha h(x) = \mathfrak D^kh(x)(e_{\sigma(1)},\ldots, e_{\sigma(k)}),
\]
where $e_1,\ldots,e_n$ denote the standard basis of $\R^n$. An inductive application of the chain rule yields a convenient formula for total derivatives
\[
\mathfrak D^k \tilde h(x)(v_1,\ldots, v_k) = \mathfrak D^k h(Mx)(Mv_1,\ldots, Mv_k)
\]
for any vectors $v_1,\ldots, v_k \in \R^n$. Consequently, we have
\[
\| \mathfrak D^k \tilde h(M^{-1} x)\| \le \|M\|^k \|\mathfrak D^k h(x)\| \qquad\text{for }x\in U.
\]
Since all norms in finite dimensional space are equivalent we have
\[
\sup_{|\beta|=k} |\partial^\beta h(x)| \sim \|\mathfrak D^k h(x)\|.
\]
Since the same equivalence holds for $\tilde h$,
\[
\sup_{|\beta|=k}  |\partial ^\beta \tilde h(M^{-1}x)|  \ls \|M\|^k \|\mathfrak D^k h(x)\|
\ls ||M||^k \sup_{|\beta|=k} |\partial^\beta h(x)|,
\]
and \eqref{e4.8} follows.
\end{proof}

\begin{example}\label{r4.5}
Consider a $n\times n$ real matrix $A$ with eigenvalues $\lz$ satisfying $|\lz| > 1$. By \cite[Lemma 2.2]{b}, there exists an ellipsoid $\Delta:=\{x\in\rn:\,|Px|<1\}$, where $P$ is some invertible $n\times n$ matrix, such that $B_k\subset B_{k+1}$, where $B_k:=A^k\Delta$ for $k\in \zz$. Moreover, the volume $|B_k|=b^k$,  where $b:=|\det A|$. Then we can define a semi-continuous ellipsoid cover in the sense of \cite[Definition 2.5]{ddp} by
\begin{equation}\label{ex1}
 \Theta:=\{ \theta_{x,\,- k \log_2 b}:=x+A^kP^{-1}(\mathbb{B}^n): x\in\rn,k\in\mathbb{Z}\}=\{x+B_k:\,x\in\rn,k\in\mathbb{Z}\}.
\end{equation}
We can easily turn $\Theta$ into continuous ellipsoid cover by setting for all $x\in \R^n$, $t\in \R$,
\[
\theta_{x,t} = x+ B_k, \qquad\text{where } k = - \lfloor t/ \log_2 b \rfloor.
\]
 According to \cite[Definition 2.3]{b} a homogeneous quasi-norm $\rho_A$ associated with expansive dilation $A$ is defined by \begin{equation*}
\rho_A(x):=\sum_{k\in\zz}b^k\textbf{1}_{B_{k+1}\setminus B_k}(x) \qquad\text{for }x\in \R^n,
\end{equation*}
    where $b:=|\det A|$. Let $\rho_{\Theta}$ be the quasi-distance corresponding to the cover $\Theta$ given by \eqref{e2.3}. A simple calculation shows that
 \begin{equation}\label{ex4}
\rho_{\Theta}(x,y)= \rho_\Theta ((x-y)/2,-(x-y)/2) = b \rho_A((x-y)/2) \qquad\text{for all }x,y \in \R^n.
\end{equation}
In other words, for any $k\in \Z$,
\begin{equation}\label{ex8}
\rho_\Theta(x,y)=b^k \iff \frac{x-y}2 \in B_k \setminus B_{k-1}.
\end{equation}
Thus, $\rho_\Theta$ is a translation invariant quasi-distance on $\R^n$, which reflects invariance of ellipsoid cover $\Theta$ under translations.

By Definition \ref{d4.4}, \eqref{ex1}, and \eqref{ex8}, a kernel $K$ of VASIO of order $s$ associated to the ellipsoid cover $\Theta$ satisfies for any $(x,y)\in\Omega$,
    \begin{equation*}
 \lf|\partial^\alpha_y[K(\cdot,A^{k}P^{-1}\cdot)](x,PA^{-k}y)\r|\le C/\rho_\Theta(x,y)= Cb^{-k},\ \ \ \ \ |\alpha|\le s,
    \end{equation*}
  where integer $k\in\mathbb{Z}$ satisfies $x-y\in 2(B_{k}\setminus B_{k-1})$.
Hence, by \eqref{ex4} and Lemma \ref{l4.7}, the kernel $K$ coincides with the kernel of Calder\'{o}n-Zygmund operators of order $s$ associated with expansive dilation $A$
  \begin{equation}\label{e4.3}
   \lf|\partial^\alpha_y[K(\cdot,A^{k}\cdot)](x,A^{-k}y)\r|\le C/\rho_A(x-y)= Cb^{-k},\ \ \ \ \ |\alpha|\le s,
    \end{equation}
where integer $k\in\mathbb{Z}$ satisfies $x-y\in B_{k}\setminus B_{k-1}$, see  \cite[Definition 9.2]{b}.

When $A=\lz\rm{I}$ with $|\lz|>1$ we can take an equivalent quasi-norm $\rho$ given by $\rho(x)=|x|^n$. Then, it is not difficult to show that the kernel of a Calder\'{o}n-Zygmund operator of order $s$ satisfying \eqref{e4.3} coincides with the classical definition, see \cite[Section III.7]{gf}, \cite[Section 7.4]{mc97}, and \cite[Section III.3]{s}. That is, there exists $C>0$ such that for any $(x,y)\in\Omega$,
   \begin{equation*}
   \lf|\partial^\alpha_yK(x,y)\r|\le C|x-y|^{-n-\alpha},\ \ \ \ \ |\alpha|\le s.
   \end{equation*}
\end{example}

Next we show that the definition of singular integral operators in the setting of continuous ellipsoid covers is consistent with the Calder\'on-Zygmund operators on spaces of homogenous type, see Proposition \ref{p2.4}.

\begin{proposition}\label{p4.6}
Let $K$ be kernel of VASIO of order $1$. Then, there exist positive constants $\delta$ and $C$ such that for all $x \ne y  \in \R^n$ we have
\begin{align}\label{e4.6}
|K(x,y)|&\le C/\rho_\Theta(x,y),
\\
\label{e4.7}
|K(x,y)-K(x,y')| &\le C \frac{[\rho_\Theta(y,y')]^\delta}{[\rho_\Theta(x,y)]^{1+\delta}}
\qquad\text{if \ } \rho_\Theta(y,y') \le \frac{1}{2\kappa} \rho_\Theta(x,y).
\end{align}
In particular, the kernel $K$ satisfies \eqref{e4.1}.
\end{proposition}

\begin{proof}
By \eqref{e4.2} we have \eqref{e4.6}. Next we prove \eqref{e4.7}. For a fixed $x,y\in\rn$ with $x\neq y$, let
$r:= (\kappa + 1) \rho_\Theta(x,y)$. By Proposition \ref{p4.10} there exists $m\in \R$ such that
\begin{equation}\label{e4.35}
 B_{\rho_\Theta}(x,r)
\subset\theta_{x,m}
\qquad\text{and}\qquad
2^{-m} \sim |\theta_{x, m}| \sim r.
\end{equation}
Define a rescaled kernel $\wz K(u ,v ):= K(u ,M_{y,\,m} v)$, $u,v \in\R^n$.

Take any $y' \in \R^n$ such that $\rho_\Theta(y,y') \le \frac{1}{2\kappa} \rho_\Theta(x,y)$.
By Lagrange's mean value theorem, there exists some $\xi$ on the segment between $y$ and $y'$ such that
\begin{align*}
|K(x,y)-K(x,y')|&=\lf|\wz K(x,M^{-1}_{y,\,m}y)-\wz K(x,M^{-1}_{y,\,m}y')\r|\\
&=\lf|\sum_{|\alpha|=1}\partial^\alpha_y\wz K\lf(x,M^{-1}_{y,\,m}\xi\r)(M^{-1}_{y,\,m}y-M^{-1}_{y,\,m}y')^\alpha\r|\nonumber\\
&\ls\sup_{|\alpha|=1}\lf|\partial^\alpha_y [K(\cdot,M_{y,\,m}\cdot)](x,M^{-1}_{y,\,m}\xi)\r|\lf|M^{-1}_{y,\,m}(y-y')\r|.\nonumber
\end{align*}
Let $l:=- \log_{2} \rho_\Theta(x, \xi)$. By Lemma \ref{l4.7} and Definition \ref{d4.4}, we have
\begin{align}\label{e4.13}
&|K(x,y)-K(x,y')|\\
&\hs\ls\sup_{|\alpha|=1}\lf|\partial^\alpha_y [K(\cdot,M_{\xi,\,l }M_{\xi,\,l}^{-1} M_{y,\,m}\cdot)]
(x,(M^{-1}_{\xi,\,l }M_{y,\,m})^{-1}(M_{\xi,\,l})^{-1}\xi)\r|
\lf|M^{-1}_{y,\,m}(y-y')\r|\nonumber\\
&\hs\ls\|M^{-1}_{\xi,\,l }M_{y,\,m}\|\sup_{|\alpha|=1}|\partial^\alpha_y [K(\cdot,M_{\xi,\,l }\cdot)](x,M^{-1}_{\xi,\,l }\xi)|\lf|M^{-1}_{y,\,m}(y-y')\r|\nonumber\\
&\hs\ls\|M^{-1}_{\xi,\,l }M_{y,\, m}\|
\frac1{\rho_\Theta(x,\xi)}\lf|M^{-1}_{y,\,m}(y-y')\r|.
\nonumber
\end{align}

Since $\rho_\Theta(y,\xi)\le\rho_\Theta(y,y')$ (convexity of ellipsoids) we have
\begin{align*}
\rho_\Theta(x,y)&\le\kappa[\rho_\Theta(x,\xi)+\rho_\Theta(y,\xi)]
\le\kappa[\rho_\Theta(x,\xi)+\rho_\Theta(y,y')]\\
&\le\kappa\lf[\rho_\Theta(x,\xi)+\frac1{2\kappa}\rho_\Theta(x,y)\r].
\end{align*}
Hence,
\begin{equation}\label{e4.14}
\rho_\Theta(x,y)\le 2\kappa\rho_\Theta(x,\xi).
\end{equation}
Likewise,
\[
 \rho_\Theta(x,\xi) \le \kappa [\rho_\Theta(x,y)+\rho_\Theta(y,\xi)] \le (\kappa + 1/2) \rho_\Theta(x,y).
\]
Since $\xi \in B_{\rho_\Theta}(x,r) \subset \theta_{x,\, m}$ we have $\theta_{x, \, m} \cap \theta_{\xi, \, l} \ne \emptyset$.
Since $ \rho_\Theta(x,y) \sim  \rho_\Theta(x,\xi)$ we have $2^{-m} \sim 2^{-l}$ and hence $|m-l| \ls 1$. By \eqref{e2.2} we deduce that
\begin{equation}\label{e4.40}
\|M^{-1}_{\xi,\,l }M_{y,\, m}\| \le
\begin{cases}
a_5 2^{-a_6(m-l)}  & m \ge l,\\
(1/a_3) 2^{a_4(l-m) }& l \ge  m,
\end{cases}
\end{equation}
and hence $\|M^{-1}_{\xi,\,l }M_{y,\, m}\| \ls 1$.

By Lemma \ref{l4.8} there exists $k\in\mathbb{Z}$ such that $y'\in\theta_{y,\,k\gamma}\setminus\theta_{y,\,(k+1)\gamma}$. By Proposition \ref{p4.9} and \eqref{e4.35}
\[
  2^{-k\gamma} \sim \rho_1(y,y') \sim \rho_\Theta(y,y') \ls \rho_\Theta(x,y) \sim 2^{-m}.
\]
This implies that there exists a constant $\eta>0$
$$m-k\gamma\le\eta.$$
Since $M^{-1}_{y,\,k\gamma}(y-y')\in\mathbb{B}^n$, the property \eqref{e2.2} implies that
\begin{align*}
|M^{-1}_{y,\,m}(y-y')|&=|M^{-1}_{y,\,m}
M_{y,\,k\gamma}M^{-1}_{y,\,k\gamma}(y-y')|
\le\|M^{-1}_{y,\,m}M_{y,\,k \gamma}\|  |M^{-1}_{y,\,k\gamma}(y-y')|\\
&=\|M^{-1}_{y,\,m}M_{y,\,k\gamma+\eta}
M^{-1}_{y,\,k\gamma+\eta}M_{y,\,k \gamma}\|  |M^{-1}_{y,\,k\gamma}(y-y')|\\
&\le\|M^{-1}_{y,\,m}M_{y,\,k\gamma+\eta}\|
\|M^{-1}_{y,\,k\gamma+\eta}M_{y,\,k \gamma}\|  |M^{-1}_{y,\,k\gamma}(y-y')|\\
&\ls 2^{-a_6(k\gamma -m)}
\sim \frac{[\rho_\Theta(y,y')]^{a_6}}{[\rho_\Theta(x,y)]^{a_6}}.
\end{align*}
Combining this with \eqref{e4.13}, \eqref{e4.14}, and \eqref{e4.40} yields
$$|K(x,y)-K(x,y')|\ls\frac{1}{\rho_\Theta(x,y)}
\frac{[\rho_\Theta(y,y')]^{a_6}}{[\rho_\Theta(x,y)]^{a_6}}
=\frac{[\rho_\Theta(y,y')]^{a_6}}
{[\rho_\Theta(x,y)]^{1+a_6}}.$$
Therefore, \eqref{e4.7} holds with $\delta=a_6$.

Finally, the fact that $K$ satisfies \eqref{e4.1} follows from general results for spaces of homogeneous type. More precisely, we claim that \eqref{e4.1} with the constant $c=2\kappa$.
Indeed, take $y'\in B_{\rho_\Theta}(y,r)$ for some $r>0$. By \eqref{e4.7} and Proposition \ref{p2.4}
we have
\begin{align*}
\int_{B_{\rho_\Theta}(y,\,2\kappa r)^\complement}|K(x,y)-K(x,y')|dx
&\ls \int_{B_{\rho_\Theta}(y,\,2\kappa r)^\complement}
\frac{r^\delta}{[\rho_\Theta(x,y)]^{1+\delta}}dx\\
&=\sum_{i=1}^\fz\int_{B_{\rho_\Theta}(y,\,2^{i+1}\kappa r)
\setminus{B_{\rho_\Theta}(y,\,2^{i}\kappa r)}}
\frac{r^\delta}{[\rho_\Theta(x,y)]^{1+\delta}}dx\\
&\le \sum_{i=1}^\fz\frac{r^{\delta}}{(2^i\kappa r)^{(1+\delta)}}
|B_{\rho_\Theta}(y,2^{i+1}\kappa r)|\\
&\sim \sum_{i=1}^\fz 2^{-i\delta}\ls 1.
\end{align*}
\end{proof}

The following lemma is a convenient strengthening of Definition \ref{d4.4}.

\begin{lemma}\label{l4.11}
Suppose that $T$ is a VASIO of order $s$ as in Definition \ref{d4.4}. Then, there exists a constant $C>0$ such that for any
$z\in \R^n$, $t\in \R$, $k\in \nn$,  $x\in\theta_{z,\,t-(k+1)\gamma}\setminus\theta_{z,\,t-k\gamma}$, and $y\in\theta_{z,\,t}$, we have
\begin{equation*}
\lf|\partial^\alpha_y[K(\cdot,M_{z,\,t-k\gamma}\cdot)](x,M^{-1}_{z,\,t-k\gamma}y)\r|\le C2^{t-k\gamma}\qquad {\rm for}\ |\alpha|\le s.
\end{equation*}
Here, $\gamma$ is as in Lemma \ref{l2.2}(ii) and the constant $C$ depends only on $\|T\|_{(s)}$ as in Definition \ref{d4.4} and $\mathbf{p}(\Theta)$ as in Definition \ref{d2.1}.
\end{lemma}

\begin{proof}
By Lemma \ref{l4.12} we have $x\in \theta_{y,\, t-(k+2) \gamma}\setminus\theta_{y,\,t-(k-1) \gamma}$ and
\begin{equation}\label{e4.20}
\rho_\Theta(x,\,z)\sim\rho_{\Theta}(x,\,y) \sim 2^{-t+k\gamma}.
\end{equation}
By Definition \ref{d4.4} we have
\begin{equation}
\lf|\partial^\alpha_y[K(\cdot,M_{y,\,m}\cdot)](x,M^{-1}_{y,\,m}y)\r|\le C/\rho_\Theta(x,y)\le C2^{m}, \label{e4.30}
\end{equation}
where $m=-  \log_2 \rho_\Theta(x,y) $. From \eqref{e4.20} it follows that $2^{-m} = \rho_\Theta(x,y) \sim 2^{-t+k\gamma}$. Hence, there exists a constant $\eta>0$
\begin{equation}\label{e4.21}
|m-(t-k\gamma)|\le \eta.
\end{equation}
Define $M:= M^{-1}_{y,\,m}M_{z,\,t-k\gamma}$. Since $\theta_{y,\, m} \cap \theta_{z,\, t- k\gamma} \ne \emptyset$, by \eqref{e2.2} and  \eqref{e4.21} we deduce that
\[
||M|| \le
\begin{cases}
a_5  & m \le t-k\gamma,\\
2^{a_4\eta}/a_3 & t-k\gamma \le m.
\end{cases}
\]
Hence, by Lemma \ref{l4.7}, \eqref{e4.30}, and \eqref{e4.21}, we conclude that for $|\alpha|\le s$,
\begin{align*}
&\lf|\partial^\alpha_y[K(\cdot,M_{z,\,t-k\gamma}\cdot)](x,M^{-1}_{z,\,t-k\gamma}y)\r|\\
&\hs=\lf|\partial^\alpha_y[K(\cdot,M_{y,\,m} M^{-1}_{y,\,m} M_{z,\,t-k\gamma}\cdot)](x,(M^{-1}_{y,\,m} M_{z,\,t-k\gamma})^{-1} M_{y,\,m}^{-1}
y)\r|\\
&\hs\ls ||M_{y,\,m}^{-1} M_{z,\,t-k\gamma}\|^{|\alpha|}\lf|\partial^\alpha_y[K(\cdot,M_{y,\,m}\cdot)]
(x,M^{-1}_{ y,\,m}y)\r|\\
&\hs\le C (2^{a_4\eta}/a_3 )^s 2^{m}
\le C 2^{\eta} (2^{a_4\eta}/a_3 )^s  2^{t-k\gamma}.
\end{align*}
This finishes the proof of Lemma \ref{l4.11}.
\end{proof}

Our ultimate goal is to show that anisotropic Calder\'{o}n-Zygmund operators $T$ are bounded on $H^p(\Theta)$. Generally, we can not expect this unless we also assume that $T$ preserves vanishing moments. Hence, we adopt the following definition motivated by \cite[Definition 9.4]{b}.

\begin{definition}\label{d4.13}
Let $s\in\nn$ and $1<q<\infty$. We say that a VASIO $T$ of order $s$ satisfies
\begin{equation*}
T^*(x^\alpha)=0 \qquad\text{for all }
|\alpha|\le l,
\end{equation*}
where $l<a_6s /a_4$, if for any $f\in L^q$ with compact support with vanishing moments $\int_\rn f(x)x^\alpha dx=0$ for all $|\alpha|<s$, we have
$$\int_\rn Tf(x)x^\alpha dx=0 \qquad\text{for all }
|\alpha|\le l.$$
\end{definition}

When the continuous ellipsoid cover $\Theta$ comes from an expansive dilation $A$ as in Example \ref{r4.5}, then Definition \ref{d4.13} overlaps with \cite[Definition 9.4]{b}. In the isotropic setting, when $A=\lambda \rm{I}$, $\lambda>1$, this definition coincides with the analogous property of vanishing moments of $T$ investigated by Coifman and Meyer in \cite[Chapter 7.4]{mc97}.

The actual value of $q$ is not relevant in Definition \ref{d4.13} as we merely need that $T:L^q \to L^q$ is bounded.
However, the requirement that $l< a_6s/a_4$ is essential to guarantee that the integrals $\int_\rn Tf(x)x^\az\,dx$ are well defined for all $|\alpha|\le l$. This is a consequence of the following lemma.

\begin{lemma}\label{decay}
Let $l,s\in\nn$, $1<q<\infty$. Let $T$ be a {\rm VASIO} of order $s$. Suppose that $f\in L^q$ satisfies $\supp f \subset \theta_{z,t}$ for some $z\in \R^n$, $t\in \R$, and $\int_\rn f(x)x^\alpha dx=0$ for all $|\alpha|<s$. Then, for some $C>0$ depending only $\|T\|_{(s)}$  and $\mathbf{p}(\Theta)$,
\begin{equation}\label{decay1}
|Tf(x)|\le C ||f||_q |\theta_{z,\,t}|^{-1/{q}}2^{-k\gamma(1+a_6s)}
\qquad\text{for }x\in \theta_{z,\,t-(k+1)\gamma}\setminus\theta_{z,\,t-k\gamma},\ k\in \nn.
\end{equation}
In particular, if $l< a_6s/a_4$, then
\begin{equation}\label{decay2}
\int_\rn|Tf(x)|(1+|x|^l) \,dx<\infty.
\end{equation}
\end{lemma}

\begin{proof}
Take any $x\in\theta_{z,\,t-(k+1)\gamma}\setminus\theta_{z,\,t-k\gamma}$, $k\in \nn$, and $y\in\theta_{z,\,t}$.
Define the rescaled kernel $\wz{K}(u,v):=K(u,M_{z,\,t-k\gamma}v)$, $u,v\in\R^n$.
By Lemma \ref{l4.11}, we have, for all $\alpha\in\nn_0^n$ with $|\alpha|\le s$,
\begin{equation} \label{e6.13}
|\partial^\alpha_y[\wz{K}(x,M^{-1}_{z,\,t-k\gamma}y)|\le C2^{t-k\gamma}.
\end{equation}
Since  $\supp f\subset\theta_{z,\,t}$, we can write
\begin{equation}\label{e6.14}
Tf(x)=\int_{\theta_{z,\,t}}K(x,y)f(y)dy=\int_{\theta_{z,\,t}}
\wz{K}(x,M^{-1}_{z,\,t-k\gamma}y)f(y)dy.
\end{equation}
 Now we expand $\wz{K}$ into the Taylor polynomial of degree $s-1$ (only in $y$ variable) at the point $(x,M^{-1}_{z,\,t-k\gamma}z)$, that is
\begin{equation}\label{e6.15}
\wz{K}(x,\wz{y})=\sum_{|\alpha|\le s-1}\frac{\partial^\alpha_y\wz{K}(x,M^{-1}_{z,\,t-k\gamma}z)}{\alpha !}(\wz{y}-M^{-1}_{z,\,t-k\gamma}z)^{\alpha}+R_s(\wz{y}),
\end{equation}
where $\wz{y}:=M^{-1}_{z,\,t-k\gamma}y$ and $y\in\theta_{z,\,t}$. Then, using \eqref{e6.13} and \eqref{e2.2}, we see that the remainder term $R_s$ satisfies
\begin{align}\label{e6.16}
|R_s(\wz{y})|&\le C\sup_{\xi\in \theta_{z,\,t}}
\sup_{|\alpha|=s}|\partial^\alpha_y\wz{K}(x,M^{-1}_{z,\,t-k\gamma}\xi)|
|\wz{y}-M^{-1}_{z,\,t-k\gamma}z|^{s}\\
&\le C 2^{t-k\gamma}\sup_{w\in\mathbb{B}^n}|M^{-1}_{z,\,t-k\gamma}M_{z,\,t}w|^{s}\nonumber\\
&\le C 2^{t-k\gamma(1+a_6s)}.\nonumber
\end{align}
Moreover, by H\"{o}lder's inequality we have
\begin{equation}\label{e6.17}
\int_{\theta_{z,\,t}}|f(y)|dy\le \|f\|_q|\theta_{z,\,t}|^{1/{q'}} \ls 2^{-t/q'} \|f\|_q,
\end{equation}
where $1/q+1/q'=1$. Finally, using \eqref{e6.14}, \eqref{e6.15}, the vanishing moments of $f$ up to order $s-1$, \eqref{e6.16} and \eqref{e6.17}, we obtain that
\begin{align*}
|Tf(x)|&\le \int_{\theta_{z,\,t}}|R_s(M^{-1}_{z,\,t-k\gamma}y)f(y)|dy\le C 2^{-k\gamma(1+a_6s)} 2^{t/q} \|f\|_q,
\end{align*}
which implies \eqref{decay1}.

To show the second part \eqref{decay2} we first choose $k_0\in\nn$ large enough such that for any $x\in(\theta_{z,t-k_0\gamma})^\complement$, we have $\rho_\Theta(x,z)\ge 1$. Then, we split the integral into two parts
 $$\int_\rn|Tf(x)|(1+|x|^l) \,dx= \bigg( \int_{\theta_{z,t-k_0\gamma}}
+\int_{(\theta_{z,t-k_0\gamma})^\complement} \bigg) |Tf(x)|(1+|x|^l) \,dx =:{\rm I}+{\rm II}.$$
The first integral is bounded by H\"older's inequality and the boundedness of $T:L^q \to L^q$,
\begin{align*}
{\rm I}&\ls C \int_{\theta_{z,t-k_0\gamma}}|Tf(x)|\,dx\le C\lf\{\int_{\theta_{z,\,t-k_0\gamma}}|Tf(x)|^q\,dx\r\}^{\frac1q} |\theta_{z,\,t-k_0\gamma}|^{\frac1{q'}}<\fz.
\end{align*}

By Lemma \ref{l2.6} and Proposition \ref{p4.9} we have
\[
|x-z| \ls \rho_\theta(z,x)^{a_4} \sim 2^{(-t+k\gamma)a_4}
\qquad\text{for } x\in\theta_{z,t-(k+1)\gamma}\setminus
\theta_{z,t-k\gamma}, \ k\ge k_0.
\]
Hence,
\[
\int_{\theta_{z,t-(k+1)\gamma}\setminus
\theta_{z,t-k\gamma}}  |x-z|^l\,dx
\ls |\theta_{z, t-(k+1)\gamma}|  2^{(-t+k\gamma)l a_4} \ls 2^{(-t+k\gamma)(l a_4+1)}.
\]
We estimate the second integral using \eqref{decay1},
\begin{align*}
{\rm II}&=\sum_{k=k_0}^{\fz}\int_{\theta_{z,t-(k+1)\gamma}\setminus
\theta_{z,t-k\gamma}}|Tf(x)|(1+|x|^l)\,dx\\
&\ls  ||f||_q |\theta_{z,\,t}|^{-1/{q}}   \sum_{k=k_0}^{\fz} 2^{-k\gamma(1+a_6s)} \int_{\theta_{z,t-(k+1)\gamma}\setminus
\theta_{z,t-k\gamma}}  (1+|z|^l+|x-z|^l)\,dx
\\
&\ls  ||f||_q |\theta_{z,\,t}|^{-1/{q}}   \sum_{k=k_0}^{\fz} 2^{-k\gamma(1+a_6s)}( 2^{-t+k\gamma} + 2^{(-t+k\gamma)(l a_4+1)})
\\
&\ls  ||f||_q |\theta_{z,\,t}|^{-1/{q}}   \sum_{k=k_0}^{\fz} ( 2^{-t-k\gamma a_6s}+ 2^{-t(l a_4+1)+k\gamma(la_4-sa_6)}) <\infty.
\end{align*}
The last series converges since we have assumed that $l a_4 -  a_6s<0$.
\end{proof}

We are now ready to state the main results of the paper, Theorems \ref{t4.17} and \ref{t4.16}. These are generalizations  \cite[Theorems 9.8 and 9.9]{b} to Hardy spaces with pointwise variable anisotropy.

\begin{theorem}\label{t4.17} Let $\Theta$ be an ellipsoid cover with parameters $\mathbf{p}(\Theta)=\{a_1,\ldots,a_6\}$ and $0<p \le 1$. Suppose that $T$ is a {\rm VASIO} of order $s$ such that
\begin{align}\label{e4.23}
s> \frac{a_4 }{a_6} N_p(\Theta)
% \textcolor{red}{= \frac{a_4(\max(1,a_4)n+1)}{(a_6)^2p} },
&\qquad\text{where } N_p(\Theta)=\bigg\lfloor \frac{\max(1,a_4)n+1}{a_6p} \bigg\rfloor+1,
\\
T^*(x^\az)=0 & \qquad\text{for all } \alpha\in\nn_0^n,\ |\az|\le  N_p(\Theta).
\label{e4.24}
\end{align}
Then, $T$ extends to a bounded linear operator from $H^p(\Theta)$ to itself.
\end{theorem}

\begin{theorem}\label{t4.16}
Let $\Theta$ be an ellipsoid cover with parameters $\mathbf{p}(\Theta)=\{a_1,\ldots,a_6\}$ and $0<p\le 1$. Suppose $T$ is a {\rm VASIO} of order $s$ with
\begin{equation}\label{e4.22}
s>\frac{1/p-1}{a_6}.
\end{equation}
Then, $T$ extends to a bounded linear operator from $H^p(\Theta)$ to $L^p$.
\end{theorem}

%%%%%%%%%%%%%%%%%%%%%%%%%%%%%%%%%%%%%%%%%%%%%%%%%

%%%%%%%%%%%%%%%%%  Section 6 %%%%%%%%%%%%%%%%%%%%

%%%%%%%%%%%%%%%%%%%%%%%%%%%%%%%%%%%%%%%%%%%%%%%%%

\section{Proofs of Theorems \ref{t4.17} and \ref{t4.16}}\label{s6}

To prove Theorems \ref{t4.17} and \ref{t4.16}, we need the following definition and some lemmas.

\begin{definition}
For $l\in\nn_0$, let $\cp_l$ denote the {\it{linear space of polynomials of degree  $\le l$}}. For an ellipsoid $\theta\subset \R^n$, let $\pi_{\theta}: L^1(\theta)\to \cp_l$ be the natural projection defined, via the Riesz lemma, for all $f\in L^1(\theta) $ and $ Q\in\cp_l$,
\begin{align}\label{e6.1}
\int_{\theta} \pi_{\theta}f(x)Q(x)\,dx=\int_{\theta} f(x)Q(x)\,dx.
\end{align}
\end{definition}

\begin{lemma}\label{l6.2}
For any $l\in\nn_0$,  there exists a positive
constant $C$ depending only on $l$ such that for any ellipsoid $\theta \subset \R^n$ and $f\in L^1(\theta)$,
\begin{align}\label{e6.2}
\sup_{x\in \theta}|\pi_{\theta} f(x)|\le C \frac1{|\theta|}\int_{\theta} |f(x)|dx.
\end{align}
\end{lemma}

\begin{proof}
Let $\theta :=M(\mathbb{B}^n)+z$ for some invertible matrix $M$ and $z\in\rn$. For a fixed $l\in\nn_0$, we choose an orthonormal basis $\{Q_\alpha:\, |\alpha|\le l\}$ of $\cp_l$ with respect to the $L^2(\mathbb{B}^n)$ norm. Since
$$\pi_{\mathbb{B}^n}f=\sum_{|\alpha|\le l}\lf(\int_{\mathbb{B}^n}f(x)\overline{Q_\alpha(x)}dx\r)Q_\alpha  \qquad {\rm for\ any}\ f\in L^1(\mathbb{B}^n),$$
we conclude that that there exists $C>0$ such that
\begin{equation}\label{e6.3}
\lf|\sup_{x\in\mathbb{B}^n}\pi_{\mathbb{B}^n}f(x)\r|\le \sup_{x\in\mathbb{B}^n}\sum_{|\alpha|\le l}\lf(\int_{\mathbb{B}^n}\lf|f(y)\overline{Q_\alpha(y)}\r|dy\r)|Q_\alpha(x)|
\le C
\frac1{|\mathbb{B}^n|}\int_{\mathbb{B}^n} |f(x)|dx.
\end{equation}
Our goal is to show that \eqref{e6.2} holds with the same constant $C$ for the ellipsoid $\theta$.
We claim that
\begin{equation}\label{e6.4}
\pi_{\theta-z} f=(D_{M^{-1}}\circ\pi_{\mathbb{B}^n}\circ D_{M})f,
\qquad
\text{where } D_{M}f(x):=f(Mx), x\in \R^n.
\end{equation}
Indeed, for any $Q\in\cp_l$, by \eqref{e6.1}, we have
\begin{align*}
&\int_{\theta-z}(D_{M^{-1}}\circ\pi_{\mathbb{B}^n}\circ D_{M})f(x)Q(x)dx
= |\det M| \int_{\mathbb{B}^n}\pi_{\mathbb{B}^n}(D_M f)(x)
Q(M x)d x\\
&\hs=|\det M| \int_{\mathbb{B}^n}f(Mx)Q(Mx)dx=\int_{\theta-z}f(x)Q(x)dx
=\int_{\theta-z}\pi_{\theta-z}f(x)Q(x)dx.
\end{align*}
From \eqref{e6.3} and \eqref{e6.4}, it follows that, for any $f\in L^1(\mathbb{B}^n)$,
\begin{align}\label{e6.5}
\sup_{x\in \theta-z}|\pi_{\theta-z} f(x)|&=\sup_{x\in \theta-z}|(D_{M^{-1}}\circ\pi_{\mathbb{B}^n}\circ D_M f(x)|
=\sup_{x\in \mathbb{B}^n}|\pi_{\mathbb{B}^n}(D_M f)(x)| \\
&\le C\frac1{|\mathbb{B}^n|}\int_{\mathbb{B}^n}|f(M x)|dx
= C\frac1{|\theta-z|}\int_{\theta-z}|f(x)|dx.\nonumber
\end{align}
Likewise, we claim that for $z\in\rn$,
\begin{equation}\label{e6.6}
\pi_{\theta}f=(\tau_z\circ\pi_{\theta -z}\circ \tau_{-z})f,
\qquad\text{where }\tau_zf(x):=f(x-z),\ x\in\rn.
\end{equation}
Indeed, for any $Q\in\cp_l$, by \eqref{e6.1}, we have
\begin{align*}
&\int_{\theta }(\tau_z\circ\pi_{\theta -z}\circ\tau_{-z})f(x)Q(x)dx
=\int_{\theta-z}\pi_{\theta-z}(\tau_{-z} f)(x)Q(x+z)dx\\
&\hs=\int_{\theta-z}f(x+z)Q(x+z)dx =\int_{\theta }f(x)Q(x)dx
=\int_{\theta}\pi_{\theta }f(x)Q(x)dx.
\end{align*}
Then, by \eqref{e6.5} and \eqref{e6.6},
\begin{align*}
\sup_{x\in \theta }|\pi_{\theta } f(x)|
=\sup_{x\in\theta -z}|\pi_{\theta-z}(\tau_{-z} f) (x)|
\le C\frac1{|\theta -z|}\int_{\theta -z}|f(x+z)|dx
= C\frac1{|\theta |}\int_{\theta }|f(x)|dx.
\end{align*}
This finishes the proof of Lemma \ref{l6.2}.
\end{proof}

The following lemma is a generalization of \cite[Lemma 9.3]{b}.

\begin{lemma}\label{l6.3}
Let $(p,q,l)$ be admissible as in Definition \ref{d3.6} and $\delta>a_4l+1$. Suppose that $g$ is a measurable function on $\rn$ such that
\begin{align}\label{e6.7}
\lf(\frac1{|\theta_{z,\,t}|}\int_{\theta_{z,\,t}}|g(x)|^qdx\r)^{\frac1q}\le C|\theta_{z,\,t}|^{-\frac1p}
&\qquad {\rm for\ some}\ \theta_{z,\,t}\in\Theta \ {\rm with}\ z\in\rn,\  t\in\mathbb{R},
\\
\label{e6.8}
|g(x)|\le C|\theta_{z,\,t}|^{-\frac1p}2^{-k\gamma\delta} & \qquad {\rm for}\ x\in\theta_{z,\,t-(k+1)\gamma}\setminus\theta_{z,\,t-k\gamma}\ {\rm with}\ k\in\nn_0,
\\
\label{e6.9}
\int_\rn g(x)x^\alpha dx=0 &\qquad {\rm for}\ |\alpha|\le l.
\end{align}
Then, $g\in H^p(\Theta)$ and $\|g\|_{H^p_{q,\,l}}(\Theta)\le C$, where $C>0$ is a constant  independent of $g$.
\end{lemma}

\begin{proof}
Given an ellipsoid $\theta\in\Theta$ consider the natural projection $\pi_\theta: L^1(\theta)\rightarrow\cp_l$ given by \eqref{e6.1}. Define the complementary projection $\wz{\pi}_\theta={\rm I}-\pi_\theta$, i.e., $\wz{\pi}_\theta f=f-\pi_\theta f$. By \eqref{e6.2}, we know that $\wz{\pi}_\theta$ is bounded on $L^q(\theta)$, i.e.,
\begin{equation*}
\|\wz{\pi}_\theta f\|_{L^q(\theta)}\le C_0\|f\|_{L^q(\theta)},
\end{equation*}
with the constant $C_0$ independent of $\theta\in\Theta$. Moreover,
$$\int_\theta\wz{\pi}_{\theta}f(x)x^\alpha dx=0\ \ \ {\rm for\ all}\ |\alpha|\le l.$$

We want to represent $g$ as a combination of atoms supported on $\theta_{z,\,t-j\gamma}$,  $j\in\nn_0$, where $\gamma$ is as in Lemma \ref{l2.2}(ii). Define the sequence of function $\{g_j\}_{j=0}^\fz$ by
\[
g_j={\bf 1}_{\theta_{z,\,t-j\gamma}}\wz{\pi}_{\theta_{z,\,t-j\gamma}}g.
\]
Clearly, $\supp g_j\subset\theta_{z,\,t-j\gamma}$. Since
$$\|g_0\|_q\le C_0\|g{\bf 1}_{\theta_{z,\,t}}\|_q\le C_0|\theta_{z,\,t}|^{\frac1q-\frac1p}$$
and $g_0$ has vanishing moments up to order $l$, we deduce that $g_0$ is a $C_0$ multiple of some $(p,q,l)$-atom (namely $(C_0)^{-1}g_0$).

We claim that $g_j\rightarrow g$ in $L^1$ (and hence in $\cs'$) as $j\rightarrow\fz$. It suffices to show that $\|\pi_{\theta_{z,\,t-j\gamma}}g\|_{L^1(\theta_{z,\,t-j\gamma})}\rightarrow 0$ as $j\rightarrow \fz$. Indeed, let $\{Q_\alpha:\, |\alpha|\le l\}$ be an orthonormal basis of $\cp_l$ with respect to the $L^2(\mathbb{B}^n)$ norm. By the argument used to show \eqref{e6.2} we have
\begin{align}\label{e6.10}
\pi_{\theta_{z,\,t-j\gamma}}g&=\lf(\tau_z\circ D_{M_{z,\,t-j\gamma}^{-1}}\circ\pi_{\mathbb{B}^n}\circ D_{M_{z,\,t-j\gamma}}\circ\tau_{-z}\r)g\\
&=\sum_{|\alpha|\le l}\lf(\int_{\mathbb{B}^n}D_{M_{z,\,t-j\gamma}}\circ\tau_{-z}g(x)
\overline{Q_\alpha(x)}dx\r)
\tau_z\circ D_{M_{z,\,t-j\gamma}^{-1}}Q_\alpha\nonumber\\
&=\sum_{|\alpha|\le l}\lf(\int_{\mathbb{B}^n}g(M_{z,\,t-j\gamma}x+z)\overline{Q_\alpha(x)}dx\r)
\tau_z\circ D_{M_{z,\,t-j\gamma}^{-1}}Q_\alpha\nonumber\\
&=\sum_{|\alpha|\le l}\lf(\int_{\theta_{z,\,t-j\gamma}}g(x)\overline{Q_\alpha(
M_{z,\,t-j\gamma}^{-1}(x-z))}dx\r)|\det (M_{z,\,t-j\gamma}^{-1})|
\tau_z\circ D_{M_{z,\,t-j\gamma}^{-1}}Q_\alpha.\nonumber
\end{align}
By \eqref{e6.9} and the uniform boundedness of coefficients of the polynomials $Q_\alpha(M_{z,\,t-j\gamma}^{-1}x)$ for $j\ge 0$, we also have
\begin{align*}
\lf\|\tau_z\circ D_{M_{z,\,t-j\gamma}^{-1}}Q_\alpha\r\|_{L^1(\theta_{z,\,t-j\gamma})}
&=\int_{\theta_{z,\,t-j\gamma}}\lf|\tau_z\circ D_{M_{z,\,t-j\gamma}^{-1}}Q_\alpha(x)\r|dx\\
&=\int_{\theta_{z,\,t-j\gamma}}|Q_\alpha(M_{z,\,t-j\gamma}^{-1}(x-z))|dx\\
&=|\det (M_{z,\,t-j\gamma})|\int_{\mathbb{B}^n}|Q_\alpha(x)|dx\\
&\le C|\det (M_{z,\,t-j\gamma})|
\end{align*}
and
\begin{equation*}
\int_{\theta_{z,\,t-j\gamma}}g(x)\overline{Q_\alpha(M_{z,\,t-j\gamma}^{-1}x)}dx
=-\int_{\theta_{{z,\,t-j\gamma}}^\complement} g(x)\overline{Q_\alpha(M_{z,\,t-j\gamma}^{-1}x)}dx\rightarrow 0\ \ {\rm as}\ \ j\rightarrow \fz.
\end{equation*}
From this, we conclude that
\begin{align*}
\|\pi_{\theta_{z,\,t-j\gamma}}g\|_{L^1(\theta_{z,\,t-j\gamma})}\rightarrow 0\ \ {\rm as}\ \ j\rightarrow \fz,
\end{align*}
which shows
\begin{equation}\label{e6.11}
g=g_0+\sum_{j=0}^\fz(g_{j+1}-g_j)\ \ \ {\rm in}\ L^1.
\end{equation}
In fact, we will prove that we also have convergence in $H^p(\Theta)$ by showing that $g_{j+1}-g_j$ are appropriate multiples of $(p,\fz,l)$-atoms supported on $\theta_{z,\,t-(j+1)\gamma}$. Indeed,
\begin{align}\label{e6.12}
\|g_{j+1}-g_j\|_\fz&=\|{\bf 1}_{\theta_{z,\,t-(j+1)\gamma}}\wz{\pi}_{\theta_{z,\,t-(j+1)\gamma}}g
-{\bf 1}_{\theta_{z,\,t-j\gamma}}\wz{\pi}_{\theta_{z,\,t-j\gamma}}g\|_\fz\\
&=\|{\bf 1}_{\theta_{z,\,t-(j+1)\gamma}\setminus\theta_{z,\,t-j\gamma}}g-{\bf 1}_{\theta_{z,\,t-(j+1)\gamma}}\pi_{\theta_{z,\,t-(j+1)\gamma}}
g+{\bf 1}_{\theta_{z,\,t-j\gamma}}\pi_{\theta_{z,\,t-j\gamma}}g\|_\fz\nonumber\\
&\le\|{\bf 1}_{\theta_{z,\,t-(j+1)\gamma}\setminus\theta_{z,\,t-j\gamma}}g\|_\fz+
\|{\bf 1}_{\theta_{z,\,t-(j+1)\gamma}}\pi_{\theta_{z,\,t-(j+1)\gamma}}g\|_\fz\nonumber\\
&\hs\hs+\|{\bf 1}_{\theta_{z,\,t-j\gamma}}\pi_{\theta_{z,\,t-j\gamma}}g\|_\fz\nonumber\\
&=:{\rm I}+{\rm II}+{\rm III}.\nonumber
\end{align}
For I, by \eqref{e6.8}, we have
\begin{align*}
{\rm I}&=\|{\bf 1}_{\theta_{z,\,t-(j+1)\gamma}\setminus\theta_{z,\,t-j\gamma}}g\|_\fz
\le |\theta_{z,\ t}|^{-1/p} 2^{-j\gamma\delta} \ls |\theta_{z,\,t-j\gamma}|^{-\frac1p}2^{-j\gamma(\delta-1/p)}.
\end{align*}
Since
\begin{align*}
\lf\|{\bf 1}_{\theta_{z,\,t-j\gamma}}\tau_z\circ D_{M_{z,\,t-j\gamma}^{-1}}Q_\alpha\r\|_\fz
&=\sup_{x\in\theta_{z,\,t-j\gamma}}\lf|\tau_z\circ D_{M_{z,\,t-j\gamma}^{-1}}Q_\alpha(x)\r|\\
&=\sup_{x\in\theta_{z,\,t-j\gamma}}|Q_\alpha(M_{z,\,t-j\gamma}^{-1}(x-z))|\\
&= \sup_{x\in\mathbb{B}^n}|Q_\alpha(x)|\le C_1
\end{align*}
for all $|\alpha|\le l$, then by \eqref{e6.9} and \eqref{e6.10}, we have
$$\|{\bf 1}_{\theta_{z,\,t-j\gamma}}\pi_{\theta_{z,\,t-j\gamma}}g\|_\fz\le C_1\sum_{|\alpha|\le l}\lf|\int_{\theta_{{z,\,t-j\gamma}}^\complement}g(x)\overline{Q_\alpha(
M_{z,\,t-j\gamma}^{-1}(x-z))}dx\r||\det(M_{z,\,t-j\gamma}^{-1})|.$$
Notice that $|Q_\alpha(x)|\le C_2|x|^l$ for any $x\in(\mathbb{B}^n)^\complement$ and some constant $C_2>0$. By \eqref{e2.2} and \eqref{e6.8}, we have
\begin{align*}
&\lf|\int_{\theta_{z,\,t-j\gamma}^\complement}g(x)\overline{Q_\alpha(
M_{z,\,t-j\gamma}^{-1}(x-z))}dx\r| \le C_2\int_{\theta_{z,\,t-j\gamma}^\complement}|g(x)|\lf|M_{z,\,t-j\gamma}^{-1}(x-z)\r|^ldx\\
&\hs \le C_2 |\theta_{z,\,t}|^{-\frac{1}{p}}\sum_{i=j}^\fz\int_{\theta_{z,\,t-(i+1)\gamma}
\setminus\theta_{z,\,t-i\gamma}}
2^{-i\gamma\delta}|M_{z,\,t-j\gamma}^{-1}(x-z)|^ldx\\
&\hs\le C_2 |\theta_{z,\,t}|^{-\frac{1}{p}}|\det(M_{z,\,t-j\gamma})|\sum_{i=j}^\fz 2^{-i\gamma\delta}
\int_{M_{z,\,t-j\gamma}^{-1}M_{z,\,t-(i+1)\gamma}(\mathbb{B}^n)}
|x|^{l}dx\\
&\hs\ls  |\theta_{z,\,t}|^{-\frac1p}
2^{-t+j\gamma} \sum_{i=j}^\fz 2^{-i\gamma\delta}
\|M_{z,\,t-j\gamma}^{-1}M_{z,\,t-(i+1)\gamma}\|^{l}
|\det(M_{z,\,t-j\gamma}^{-1}M_{z,\,t-(i+1)\gamma})|\\
&\hs\ls  |\theta_{z,\,t}|^{-\frac1p}
2^{-t-j\gamma(\delta-1)}
\sum_{i=j}^\fz 2^{-\gamma\delta(i-j)}
2^{a_4l \gamma (i-j)}2^{\gamma(i-j)}
\\
&\hs\ls |\theta_{z,\,t-j\gamma}|^{-\frac1p}
2^{-t-j\gamma(\delta-1-1/p)}
\end{align*}
The last series converges since $\delta>a_4l+1$. Therefore, ${\rm III}\ls |\theta_{z,\,t-j\gamma}|^{-\frac1p}2^{-j\gamma(\delta-1/p)}$.

Similarly, we also have ${\rm II}\ls |\theta_{z,\,t-(j+1)\gamma}|^{-\frac1p}
2^{-j\gamma(\delta-1/p)} $. Inserting the estimates of I, II and III into \eqref{e6.12} we conclude that for some constant $C_4>0$ we have
$$\|g_{j+1}-g_j\|_\fz\le C_4|\theta_{z,\,t-(j+1)\gamma}|^{-\frac1p}2^{-j\gamma (\delta-1/p)}.$$
Since functions $g_j$'s have vanishing moments up to order $l$, $g_{j+1}-g_j$ is a $\lambda_j$ multiple of a $(p,\fz,l)$-atom $a_j$ supported on $\theta_{z,\,t-(j+1)\gamma}$. That is, $g_{j+1}-g_j=\lambda_ja_j$ and $\lambda_j=C_42^{-j\gamma(\delta-1/p)}$. By \eqref{e6.11}, we have
$$\| g\|_{H^p_{q,l}(\Theta)}\le\lf((C_0)^p+\sum_{j=0}^\fz|\lambda_j|^p\r)^{\frac1p}
=\lf((C_0)^p+(C_4)^p\sum_{j=0}^\fz 2^{- jp\gamma(\delta-1/p)}\r)^{\frac1p} =:C<\infty.$$
The last series converges since $l\ge N_p(\Theta)$, $a_4 \ge a_6$, and hence,
\[
\delta-1/p>a_4 N_p(\Theta)-1/p > a_4 \frac{\max(1,a_4)n+1}{a_6p} -1/p >0.
\]
This finishes the proof of Lemma \ref{l6.3}.
\end{proof}

\begin{remark}\label{r6.1}
A function $g$ satisfying \eqref{e6.7}, \eqref{e6.8} and \eqref{e6.9} is referred to as a {\it molecule} localized around the ellipsoid $\theta_{z,\,t}$. Lemma \ref{l6.3} shows that a molecule $g$ belongs to $H^p(\Theta)$ with $H^p(\Theta)$ norm bounded by some constant depending only on $(p,\,q,\,l)$ and $\delta$. We also remark that our definition of molecule is more restrictive than what normally is understood as a molecule. For more properties of molecules we refer the interested readers to \cite{tw} in isotropic setting, \cite{zl} in weighted anisotropic setting, and \cite{abr} in variable anisotropic setting.

Lemma \ref{l6.3} can be deduced from \cite[Theorem 1.2]{abr}, but the verification would not be very enlightening and we opted for a direct proof. Such an argument would rely on two observations. First,  we observe that \cite[Theorem 1.2]{abr} holds under the assumption that $(p,q,m)$ is admissible and $d>a_4m+1-1/q$, since we automatically have
\[
a_4m+1-1/q > \max(1/p-1/q, a_4 n(1-1/q)).
\]
Second, a calculation shows that any function $g$ satisfying \eqref{e6.7}, \eqref{e6.8}, and \eqref{e6.9} is a $(p,q,l,d)$-molecule, as defined in \cite{abr}, for any $d$ satisfying
\[
\delta>d+1/q> a_4l +1.
\]
\end{remark}

The following lemma, which is a generalization of \cite[Lemma 9.5]{b}, shows that a VASIO preserving vanishing moments maps atoms into molecules.

\begin{lemma}\label{l6.4}
Let $s\in\nn$, $0<p\le1$, $1<q<\infty$, $T$ be a {\rm VASIO} of order $s$  satisfying $\eqref{e4.23}$ and \eqref{e4.24}. Then there exists a constant $C>0$, depending only on the Calder\'{o}n-Zygmund norm $\|T\|_{(s)}$ of $T$, such that $\|Ta\|_{H^p(\Theta)}\le C$ for every $(p, q, s-1)$-atom $a$.
\end{lemma}

\begin{proof}
Let $a$ be a $(p,q,s-1)$-atom and $\supp a\subset \theta_{z,\,t}$ with $z\in\rn$ and $t\in\mathbb R$. Since $T:L^q \to L^q$ is bounded, see Theorem \ref{t4.2} and Remark \ref{r4.3}, we have
\begin{equation*}
\lf(\int_{\theta_{z,\,t-\gamma}}|Ta(x)|^q\,dx\r)^{\frac1q} \ls ||a||_q \le |\theta_{z,\,t}|^{\frac1q-\frac1p} \ls |\theta_{z,\,t-\gamma}|^{\frac1q-\frac1p}.
\end{equation*}
By Lemma \ref{decay} for $x\in \theta_{z,\,t-(k+1)\gamma}\setminus\theta_{z,\,t-k\gamma}$, $k\in \nn$, we have
\begin{equation}\label{at2}
|Ta(x)|\ls ||a||_q |\theta_{z,\,t}|^{-1/{q}}2^{-k\gamma(1+a_6s)}
\ls |\theta_{z,\,t-\gamma}|^{-\frac1p} 2^{-k\gamma(1+a_6s)} .
\end{equation}
Hence, $Ta$ satisfies \eqref{e6.7} and \eqref{e6.8} with respect to $\theta_{z,\,t-\gamma}$ and $\delta=1+a_6s$. Furthermore, $Ta$ satisfies \eqref{e6.9} because $T^*(x^\alpha)=0$ for all $|\alpha|\le l=N_p(\Theta)$. By \eqref{e4.23} we have
$\delta>1+a_4l$.  Therefore, by Lemma \ref{l6.3}, there exists a constant $C>0$ independent of $a$ such that $\|Ta\|_{H^p_{q,\,l}(\Theta)}\le C$. By Theorem \ref{l3.8}, it follows that $\|Ta\|_{H^p(\Theta)}\ls C$.
\end{proof}

\begin{proof}[Proof of Theorem \ref{t4.17}]
Let $f\in H^p(\Theta)\cap L^q$.  By Theorem \ref{l5.9}, there exists an atomic decomposition
 \begin{equation}\label{at3}
 f=\sum_{k\in\zz}\sum_{i\in\nn_0}\lambda^k_ia^k_i \ \mathrm{ converges\ in \ } L^q
 \end{equation}
 such that $a^k_i$'s are $(p,\fz,s-1)$-atoms, and hence also $(p,q,s-1)$-atoms, and
\begin{equation}\label{e6.21}
\sum_{k\in\mathbb{Z}}\sum_{i\in\nn_0}|\lambda^k_i|^p\le C\|f\|^p_{H^p(\Theta)}.
\end{equation}
Since $T$ is bounded on $L^q$ (see Theorem \ref{t4.2} and Remark \ref{r4.3}), it follows that $Tf=\sum_{k\in\zz}\sum_{i\in\nn_0}\lambda^k_iTa^k_i$ in $L^q$ and hence
\begin{equation}\label{e6.22}
Tf=\sum_{k\in\zz}\sum_{i\in\nn_0}\lambda^k_iTa^k_i\ \ {\rm in}\ \ \cs'.
\end{equation}
Since $T$ is a {\rm VASIO} of order $s$ and $T^*(x^\alpha)=0$ for all $|\alpha|\le N_p(\Theta)$, by Lemma \ref{l6.4}, we obtain $\|Ta^k_i\|_{H^p(\Theta)}\leq C'$.
Recall that $\ell^p$ norm dominates $\ell^1$ norm for $0<p<1$. Thus, by \eqref{e6.21} and \eqref{e6.22}, we have
\begin{align*}
\|Tf\|_{H^p(\Theta)}^p&=\|M^\circ (Tf)\|_{p}^p\le\lf\|\sum_{k\in\zz}\sum_{i\in\nn_0} |\lambda^k_i| M^\circ(Ta^k_i)\r\|_{p}^p\le \sum_{k\in\zz}\sum_{i\in\nn_0}|\lambda^k_i|^p\|M^\circ(Ta^k_i)\|_p^p\\
&=\sum_{k\in\zz}\sum_{i\in\nn_0}|\lambda^k_i|^p\|Ta^k_i\|^p_{H^p(\Theta)}\le C'\sum_{k\in\zz}\sum_{i\in\nn_0}|\lambda^k_i|^p\le C' C\|f\|_{H^p(\Theta)}^p.
\end{align*}
By the density of $L^q\cap H^p(\Theta)$ in $H^p(\Theta)$, see Lemma \ref{l5.8}, and the completeness of $H^p(\Theta)$, see Lemma \ref{l5.1}(ii), we deduce that $T$ extends to a bounded linear operator from $H^p(\Theta)$ to $H^p(\Theta)$. \end{proof}

\begin{proof}[Proof of Theorem \ref{t4.16}]
Let $l :=\max(N_p(\Theta),s-1)$.
Let $a$ be a $(p,q, l)$-atom with $\supp a\subset\theta_{z,\,t}$, where $z\in\rn$ and $t\in\mathbb R$. We first show
\begin{equation}\label{e6.19}
\|Ta\|_p\leq C'.
\end{equation}
By Lemma \ref{decay} we deduce that \eqref{at2} holds for $x\in \theta_{z,\,t-(k+1)\gamma}\setminus\theta_{z,\,t-k\gamma}$, $k\in\nn$. Hence,
\begin{align}\label{e6.20}
\int_{\theta^\complement_{z,\, t- \gamma}} |Ta(x)|^pdx & =\sum_{k=1}^\fz
 \int_{\theta_{z,\,t-(k+1)\gamma} \setminus \theta_{z,\,t-k\gamma} }|Ta(x)|^pdx \\
 \nonumber
&\ls |\theta_{z,\,t-\gamma}|^{-1} \sum_{k=1}^\fz 2^{-pk\gamma(1+a_6s)} |\theta_{z,\,t-(k+1)
\gamma}|
\ls \sum_{k=1}^\fz 2^{-pk\gamma (1+a_6s- 1/p)}\ls 1.\nonumber
\end{align}
The last series converges by the assumption \eqref{e4.22}.
By the boundedness of $T$ on $L^q$, $1<q<\infty$, and H\"older's inequality
 $$
 \int_{\theta_{z,\,t-\gamma}}|Ta(x)|^pdx
 \leq
\bigg( \int_{\theta_{z,\,t-\gamma}}|Ta(x)|^qdx \bigg)^{p/q} |\theta_{z,\, t-\gamma}|^{1-p/q} \ls ||a||^p _q |\theta_{z,\, t-\gamma}|^{1-p/q} \ls 1.
 $$
This together with \eqref{e6.20} implies that \eqref{e6.19} holds true.

Next we proceed exactly as in the proof of Theorem \ref{t4.17}. By Lemma \ref{e5.9},
any $f\in L^q\cap H^p(\Theta)$ admits an atomic decomposition \eqref{at3} into $(p,q,l)$-atoms $a^k_i$'s such that \eqref{e6.21} holds and $Tf=\sum_{k\in\zz}\sum_{i\in\nn_0}\lambda^k_iTa^k_i$ in $L^q$.
Hence there exists a subsequence of the partial sum sequence $\{\sum_{k=-K}^K\sum_{i=1}^K\lambda^k_iTa^k_i\}_{K\in\nn_0}$ which converges almost everywhere to $Tf$. In that sense, we have
\begin{equation*}
Tf=\sum_{k\in\zz}\sum_{i\in\nn_0}\lambda^k_iTa^k_i \qquad {\rm almost\ everywhere}.
\end{equation*}
By the monotonicity of the $\ell^p$-norm with $0<p\le1$, \eqref{e6.21}, and \eqref{e6.19}, we deduce that for $f\in L^q\cap H^p(\Theta)$,
\begin{align*}
\|Tf\|_p^p=\lf\|\sum_{k\in\zz}\sum_{i\in\nn_0}\lambda^k_iTa^k_i\r\|_p^p
\leq
\lf\|\sum_{k\in\zz}\sum_{i\in\nn_0}|\lambda^k_i||Ta^k_i|\r\|_p^p
&\le
 \sum_{k\in\zz}\sum_{i\in\nn_0}|\lambda^k_i|^p\lf\|Ta^k_i\r\|_p^p
\\
&\le C'\sum_{k\in\zz}\sum_{i\in\nn_0}|\lambda^k_i|^p\le C' C\|f\|_{H^p(\Theta)}^p.
\end{align*}
The density of $L^q\cap H^p(\Theta)$ in $H^p(\Theta)$, see Lemma \ref{l5.8}, implies that $T$ extends to a bounded linear operator from $H^p(\Theta)$ to $L^p$.
\end{proof}

\medskip

\noindent Marcin Bownik: Department of Mathematics, University of Oregon, Eugene, OR 97403--1222, USA

\medskip

\noindent Baode Li: College of Mathematics and System Science, Xinjiang University, Urumqi, 830046, P. R. China

\medskip

\noindent Jinxia Li (Corresponding author): School of Mathematics and Information Science, Henan Polytechnic University, Jiaozuo, 454003, P. R. China

\medskip

\noindent{\it E-mail address}:

\noindent\texttt{mbownik@uoregon.edu}\quad(Marcin Bownik)

\noindent\texttt{baodeli@xju.edu.cn}\quad(Baode Li)

\noindent\texttt{jinxiali@hpu.edu.cn}\quad(Jinxia Li)


\begin{thebibliography}{30}
\vspace{-0.3cm}
\bibitem{abr} V. Almeida, J. J. Betancor and L. Rodr\'{\i}guez-Mesa, Molecules associated to Hardy spaces with pointwise variable anisotropy. Integr. Equ. Oper. Theory {\bf 89} (2017), 301--313.

\vspace{-0.3cm}
\bibitem{am}
R. Alvarado and M. Mitrea,
Hardy Spaces on Ahlfors-Regular Quasi Metric Spaces.
A Sharp Theory. Lecture Notes in Math. 2142. Springer Press, Cham (2015).


\vspace{-0.3cm}
\bibitem{bb}
B. Barrios and J. J. Betancor, Anisotropic weak Hardy spaces and wavelets. J.
Funct. Spaces Appl. Article ID 809121, {\bf 17} (2012).

\vspace{-0.3cm}
\bibitem{bd} J. J. Betancor and W. Dami\'{a}n, Anisotropic local Hardy spaces. J. Fourier Anal. Appl. {\bf 16} (2010), 658-675.

\vspace{-0.3cm}
\bibitem{b} M. Bownik, Anisotropic Hardy spaces and
wavelets. Mem. Amer. Math. Soc. {\bf 164}, no. 781 (2003), 1--122.

\vspace{-0.3cm}
\bibitem{b2}
M. Bownik, Boundedness of operators on Hardy spaces via atomic decompositions. Proc. Amer. Math. Soc. {\bf 133} (2005), 3535--3542.

\vspace{-0.3cm}
\bibitem{blyz} M. Bownik, B. Li, D. Yang and Y. Zhou,
Weighted anisotropic Hardy spaces and their applications in
boundedness of sublinear operators. Indiana Univ. Math. J. {\bf 57}
(2008), 3065--3100.

\vspace{-0.3cm}
\bibitem{ct75}
A.-P. Calder\'on and A. Torchinsky, Parabolic maximal functions
associated with a distribution. Adv. Math. {\bf 16} (1975), 1--64.

\vspace{-0.3cm}
\bibitem{ct77}
A.-P. Calder\'on and A. Torchinsky, Parabolic maximal functions
associated with a distribution. II. Adv. Math. {\bf 24} (1977), 101--171.

\vspace{-0.3cm}
\bibitem{cw71}  R. R. Coifman and G. Weiss,
Analyse Harmonique Non-commutative sur Certains Espaces Homog\`enes.
Lecture Notes in Math. 242. Springer Press, Berlin (1971).

\vspace{-0.3cm}
\bibitem{cw77}
  R. R. Coifman and G. Weiss,
Extensions of Hardy spaces and their use in analysis.
Bull. Amer. Math. Soc. {\bf 83} (1977), 569--645.


\vspace{-0.3cm}
\bibitem{ddp} W. Dahmen, S. Dekel and P. Petrushev, Two-level-split decomposition of anisotropic Besov spaces. Constr. Approx. {\bf 31} (2010), 149--194.

\vspace{-0.3cm}
\bibitem{dhp}
S. Dekel, Y. Han and P. Petrushev, Anisotropic meshless frames on $\rn$. J. Fourier Anal. Appl. {\bf 15} (2009), 634--662.

\vspace{-0.3cm}
\bibitem{dpw} S. Dekel, P. Petrushev and T. Weissblat, Hardy spaces on $\rn$ with pointwise variable anisotropy. J. Fourier Anal. Appl. {\bf 17} (2011), 1066--1107.

\vspace{-0.3cm}
\bibitem{dw}
S. Dekel and T. Weissblat, On dual spaces of anisotropic Hardy spaces. Math. Nachr. {\bf 285} (2012), 2078--2092.

\vspace{-0.3cm}
\bibitem{fs72}
 C. Fefferman and E. M. Stein, $H^p$ spaces of several variables. Acta Math. {\bf 129} (1972), 137--193.

 \vspace{-0.3cm}
 \bibitem{gf}
J. Garc\'ia-Cuerva and J. L. Rubio de Francia,
Weighted Norm Inequalities and Related Topics. North-Holland Publishing Co., Amsterdam (1985).

\vspace{-0.3cm}
\bibitem{ggkk}
I. Genebashvili, A. Gogatishvili, V. Kokilashvili and M. Krbec, Weight Theory for Integral Transforms on Spaces of Homogeneous Type. Longman Press, Harlow (1998).


\vspace{-0.3cm}
\bibitem{g}
L. Grafakos, Classical Fourier Analysis. Springer Press, New York (2014).

\vspace{-0.3cm}
\bibitem{g1}
L. Grafakos, Modern Fourier Analysis. Springer Press, New York (2014).


\vspace{-0.3cm}
\bibitem{hs}
Y. Han and E. T. Sawyer,
Littlewood-Paley theory on spaces of homogeneous type and the classical function spaces.
Mem. Amer. Math. Soc. {\bf 110} (1994), vi+126 pp.

\vspace{-0.3cm}
\bibitem{hhllyy}
Z. He, Y. Han, J. Li, L. Liu, D. Yang and W. Yuan,
A complete real-variable theory of Hardy spaces on spaces of homogeneous type. J. Fourier Anal. Appl. {\bf 25} (2019), 2197--2267.

\vspace{-0.3cm}
\bibitem{h} G. Hu, Littlewood-Paley characterization of weighted anisotropic Hardy
spaces. Taiwan. J. Math. {\bf 17} (2013), 675--700.

\vspace{-0.3cm}
\bibitem{hlyy} L. Huang, J. Liu, D. Yang and W. Yuan, Real-variable characterization of new anisotropic mixed-norm Hardy spaces. Comm. Pure Appl. Anal. (to appear).

\vspace{-0.3cm}
\bibitem{lbyz}
B. Li, M. Bownik, D. Yang and Y. Zhou,
Anisotropic singular integrals in product spaces.
Sci. China Math. {\bf 53} (2010), 3163--3178.

\vspace{-0.3cm}
\bibitem{lyy} B. Li, D. Yang and W. Yuan, Anisotropic Hardy spaces of
Musielak-Orlicz type with applications to boundedness of sublinear operators. Sci. World J. vol. 2014 (2014), Article ID 306214.

\vspace{-0.3cm}
\bibitem{msv}
S. Meda, P. Sj\"ogren and M. Vallarino,
On the $H^1$-$L^1$ boundedness of operators.
Proc. Amer. Math. Soc. {\bf 136} (2008), 2921--2931.

\vspace{-0.3cm}
\bibitem{mc97} Y. Meyer and R. R. Coifman. Wavelets:
Calder\'on-Zygmund and Multilinear Operators. Cambridge University
Press, Cambridge (1997).

\vspace{-0.3cm}
\bibitem{s} E. M. Stein, Harmonic Analysis: Real-Variable Methods, Orthogonality and Oscillatory Integrals. Princeton University Press, Princeton, NJ (1993).

\vspace{-0.3cm}
\bibitem{sw}
E. M. Stein and G. Weiss,
On the theory of harmonic functions of several variables. I. The theory of $H^p$ spaces.
Acta Math. {\bf 103} (1960), 25--62.

\vspace{-0.3cm}
\bibitem{tw}
M. Taibleson and G. Weiss,
The Molecular Characterization of Certain Hardy Spaces. Representation theorems for Hardy spaces, pp. 67--149,
Ast\'erisque, 77, Soc. Math. France, Paris (1980).

\vspace{-0.3cm}
\bibitem{w} L.-A. Wang, A multiplier theorem on anisotropic Hardy spaces. Canad. Math. Bull. {\bf 61} (2018), 390--404.

\vspace{-0.3cm}
\bibitem{zl} K. Zhao and L.-L. Li, Molecular decomposition of weighted anisotropic Hardy
spaces. Taiwan. J. Math. {\bf 17} (2013), 583--599.

\end{thebibliography}
\end{document}